\documentclass{article}

\usepackage[affil-it]{authblk}
\usepackage{graphicx}
\graphicspath{{images/}}
\usepackage{amsmath,amsthm,amssymb,amsfonts, bm, bbm, stmaryrd}
\usepackage{ben_macro}

\usepackage[utf8]{inputenc}
\usepackage[T1]{fontenc}
\usepackage{setspace}
\usepackage{fullpage}
\usepackage[backend=biber,style=numeric,giveninits, maxbibnames=9]{biblatex} 
\usepackage{tikz}
\usetikzlibrary{arrows.meta, bending}

\newcommand{\Forall}{\:\forall\:}
\newcommand{\Foreach}{\quad\Forall}

\begin{document}
\title{Discretisation of an Oldroyd-B viscoelastic fluid flow using a Lie derivative formulation}
\date{}
\author{Ben S. Ashby}

\author{Tristan Pryer}
\affil{Institute for Mathematical Innovation,\\ 
University of Bath, Bath, UK.}
\affil{Department of Mathematical Sciences, \\ University of Bath, Bath, UK.}
\maketitle

\begin{abstract}
  In this article we present a numerical method for the Stokes flow of an
  Oldroyd-B fluid. The viscoelastic stress evolves according to a constitutive
  law formulated in terms of the upper convected time derivative. A finite
  difference method is used to discretise along fluid trajectories to
  approximate the advection and deformation terms of the upper convected
  derivative in a simple, cheap and cohesive manner, as well as ensuring that
  the discrete conformation tensor is positive definite. A full implementation
  with coupling to the fluid flow is presented, along with detailed discussion
  of the issues that arise with such schemes. We demonstrate the performance of
  this method with detailed numerical experiments in a lid-driven cavity setup.
  Numerical results are benchmarked against published data, and the method is
  shown to perform well in this challenging case.
\end{abstract}

\section{Introduction}

Viscoelastic fluids produce a number of phenomena distinct from the
familiar behaviour of Newtonian fluids. A broad class of such fluids
is that of polymer solutions, for which these phenomena are the result
of long chain molecules in the fluid that are large enough to have a
meaningful effect on the fluid flow. Such fluids are found in industry
in the form of molten plastics or machine lubricants, as well as
several biological fluids such as blood. Polymeric fluids are modelled
by augmenting the usual fluid flow equations for conservation of mass
and momentum with a constitutive law for the polymeric stress, which
results in a system of either partial differential or
integro-differential equations. This article will focus on the former,
with the constitutive law given by a hyperbolic partial differential equation (PDE).

In Newtonian fluid dynamics, the \emph{Reynolds number} indicates
the relative scaling of inertial terms to viscous terms in the Navier-Stokes
equations, and can be used to predict flow behaviour and transition between different regimes such as laminar and turbulent flow. 
In viscoelastic flows, there are two dimensionless groups which are used to
predict flow behaviour, the \emph{Weissenberg number}, $\wi$ and the
\emph{Deborah number}, $\operatorname{De}$. The Weissenberg number quantifies
the effect of nonlinearities that arise due to non-Newtonian normal stress
differences \cite{dealy2010weissenberg,pakdel1997cavity}. The Deborah number has
been interpreted as the ratio of the magnitudes of elastic and viscous forces
\cite{bird1987dynamics}. In the limit as $\operatorname{De} \to 0$,
Newtonian flow behaviour is recovered. Experimental results for dilute polymer solutions
reported in \cite{pakdel1998cavity} suggest that a transition to turbulent flow behaviour may occur
when elastic effects become dominant over viscous, an effect known as elastic
instability, even for small Reynolds number. In complex flow geometries and
time-dependent flows, it may be difficult to define these dimensionless groups
appropriately, and indeed one may be more relevant than the other. However, in
keeping with much of the literature on numerical methods for non-Newtonian
flows, we will mainly consider the Weissenberg number $\wi$, defined as 
\begin{equation*}
  \wi
  :=
  \frac{\lambda U}{L},
\end{equation*}
where $\lambda$ is the relaxation time and $U$ and $L$ are characteristic
velocity and length scales of the problem.

Numerical methods posed for the solution of viscoelastic fluid
problems, especially using the Oldroyd-B model, have been observed to
suffer from the so-called high Weissenberg number problem, where
numerical simulations fail to converge at moderate values of the
Weissenberg number. There are several potential causes of the high
Weissenberg number problem such as the fact that many viscoelastic
fluid flow problems exhibit strong, spatially exponential layers which
require fine mesh resolution to adequately resolve. In
\cite{pan2009simulation}, the authors observe a relationship between
the Weissenberg number and the principle eigenvalue, determined by
numerical simulation and curve fitting. They find that this maximum
value increases exponentially with respect to $\wi$. It is therefore a
significant computational cost to resolve this layer for larger
Weissenberg number. It has also been suggested that the Oldroyd-B
model itself is insufficient to properly describe viscoelastic
behaviour. Indeed, extensional viscosity of the Oldroyd-B model blows
up at a finite strain rate \cite{hinch2021oldroyd}. Despite this, the
Oldroyd-B model remains a popular test case for numerical
modelling. This point is discussed in greater detail in
\S\ref{sec:oldroyd_b_model}. 

It is also well known that for the
Oldroyd-B model, if the initial data for the conformation tensor is
positive definite, then the conformation tensor remains so for all
time \cite{dupret1986signe, renardy2021mathematician}. Loss of
positive definiteness at the discrete level is often a precursor to
numerical blowup, and indeed it was shown in
\cite{rutkevich1970propagation} that the Upper-convected Maxwell model
(which is the model that results if the solvent contribution is
neglected in the Oldroyd-B model) becomes ill-posed. It therefore
appears that designing schemes which preserve the positive
definiteness of the discrete conformation tensor is important not only
to maintain the physical structures of the model, but also crucial for
numerical stability. 

These phenomena motivated what seems to
be the dominant method for solving viscoelastic constitutive laws in
differential form: the log-conformation representation, introduced in \cite{fattal2004constitutive}. In this
approach, a differential equation is derived for the tensor logarithm
of the (positive definite) conformation tensor, thus guaranteeing that
the numerical approximation to the conformation tensor will remain
positive definite. It is argued that this method is more appropriate
for resolving exponential spatial layers, and indeed has provided
numerical simulations with relatively large Weissenberg numbers
\cite{fattal2005time,hulsen2005flow} which are stable, though
questions of accuracy remain. It is however suggested by the results
of \cite{zhang2021comparative} that careful application of the method
is required to preserve physical characteristics of the model. In
particular, they report reduced turbulent drag reduction. In addition,
solving for the logarithm of the conformation tensor introduces
additional nonlinearities into the numerical model that are not
present in the physical model, increasing computational expense. Other schemes which preserve positivity have been derived using for example the square root of the conformation tensor \cite{lozinski2003energy} and the use of the hyperbolic tangent transformation \cite{jafari2018property}.

Schemes based on a Lie derivative formulation offer a potential solution to the
issues described above. While characteristic schemes are motivated by the idea
of discretising the material derivative as a single object after expressing it
as a rate along fluid trajectories, the upper-convected derivative is similarly
expressible as a rate along fluid paths, this time including information on
deformation of the moving fluid. Indeed, when expressed in a convected
coordinate system, the Oldroyd-B constitutive law is a tensor-valued ordinary
differential equation. The method presented in this work can be thought of as
resulting from an approximation of a rate of change of stress that also
incorporates rate of change of the convected coordinate system used to express
it. Discretising the Oldroyd-B constitutive law with an appropriately chosen
finite difference methods provides a natural way to preserve positive
definiteness of the conformation tensor. To our knowledge, the idea of
conducting computations for viscoelastic fluids by discretising the upper
convected derivative expressed as a Lie derivative was first proposed in
\cite{petera2002new}. Theory was developed in \cite{lee2006new,lee2011stable}
for a finite element method in the framework of Ricatti equations. A finite
difference scheme for the Oldroyd-B constitutive law with prescribed velocity
field was presented with truncation analysis in \cite{medeiros2021second}. These
methods can be viewed as an extension of characteristic schemes, which are well
established in fluid dynamics due to their favourable stability properties, see
\cite{pironneau1982transport} for an early reference, with analysis of a finite
element method given in \cite{suli1988convergence}. For example, when applied to
the transport equation, stable solutions are obtained even for large CFL
numbers. These methods have also been used to discretise the advective part of
the upper convected derivative in computations involving viscoelastic fluids
\cite{boyaval2009free}.

The discretisation of this problem in the Lie derivative framework has distinct
advantages and unique challenges. One of the goals of the paper is to compare
these features with other common approaches, and to test the limits of Lie
derivative discretisation methods. We present a scheme based upon the Lie
derivative where the conformation tensor is evolved pointwise, and spatial
discretisation enters only through the approximation of the deformation gradient
tensor and particle trajectories and interpolation. We note that similar schemes
have appeared in the literature
\cite{bensaada2004error,lee2006new,lee2011stable,medeiros2021second}, but our
aim in this work is to present results from a full implementation in a
challenging test case that can be compared with other published data.

In this article, we provide an accessible overview of
viscoelastic fluid dynamics, particularly focused on polymer
solutions modeled by the Oldroyd-B constitutive law. We develop a
numerically tractable scheme for the model and prove the
well-posedness. Our developed scheme also performs favorably in
comparison to more computationally intensive alternatives that
typically require high performance computing to realise. A thorough
comparative study is presented, demonstrating the robustness of our
approach. Furthermore, we discuss appropriate mesh design critical
for capturing the complex dynamics of viscoelastic flows.
This article will focus on the lid-driven cavity problem: a
challenging test problem for non-Newtonian flows due to strong
shearing that occurs near the moving lid, as well as strong stress
gradients at the corners. It is a well-studied benchmark test case in
the numerical literature on simulation of the Oldroyd-B model
\cite{boyaval2010lid,comminal2015robust,fattal2005time,habla2014numerical,
pan2009simulation,sousa2016lid}, and serves as a useful example to illustrate key differences between Newtonian and non-Newtonian flows. 

The outline of the rest of the article is as follows. In
\S\ref{sec:oldroyd_b_model}, we describe the Oldroyd-B model. Key concepts from
continuum mechanics are introduced in \S\ref{sec:lie_derivative_approximation},
where a reformulation of the upper convected time derivative is described. In
\S\ref{sec:semi_disc}, a semidiscrete upper convected time derivative is
introduced, and its consistency and approximation order is given. We discuss
fully discrete numerical methods in \S\ref{sec:numerical_schemes}. We work with
a decoupled formulation, and so we discuss solution of the Stokes and Oldroyd-B
systems separately. Descriptions of computational aspects and implementation
details, as well as numerical experiments are presented in \S\ref{sec:numerics}.

\section{The Oldroyd-B model}\label{sec:oldroyd_b_model}

Let $\W\subset \reals^d$ be a convex domain with polygonal boundary,
and let $\vec u:\W \times [0, T] \to \reals^d$ be the spatial velocity
field, with $p:\W \times [0, T] \to \reals$ the pressure. The (symmetric)
\emph{conformation tensor} $\sig:\W \to \reals^{d^2}$ describes the
non-Newtonian component of the stress. We let $\symgrad{\vec u}$ be
the rate of strain tensor, or symmetrised velocity gradient:
\begin{equation*}
  \symgrad{\vec u} = \frac 1 2 \left(\grad \vec u + (\grad \vec u)^T\right).
\end{equation*}
The Oldroyd-B model for the flow of a fluid in $\W$ consists of the
Navier-Stokes equations for incompressible flow given by
\begin{equation}\label{eq:Navier_Stokes_momentum}
  \begin{split}
    \operatorname{Re}
    \left(\frac{\partial \vec u}{\partial t}
    +
    (\vec u \cdot \grad) \vec u\right)
    &=
    - \grad p
    + 2\beta \operatorname{div}\qp{\symgrad{\vec u}}
    + \frac{1 - \beta}{\operatorname{Wi}}\operatorname{div}\qp{\sig}
    \\
    \operatorname{div} \vec u &= 0.
  \end{split}
\end{equation}
This is coupled to a hyperbolic constitutive law for the evolution of
the stress, and is given below in non-dimensional form.
\begin{equation}\label{eq:constitutive_law}
  \frac{\partial \sig}{\partial t} + (\vec u \cdot \grad) \sig 
  -
  (\grad \vec u) \sig - \sig (\grad \vec u)^T
  =
  -\frac{1}{\operatorname{Wi}}(\sig - \vec I).
\end{equation}

The system has three parameters, namely the familiar Reynolds number
$\operatorname{Re}$, the elasticity ratio $\beta$ and the
\emph{Weissenberg number}, $\wi$. The regime of interest in this work
is low $\operatorname{Re}$ and high $\wi$. We will assume that $\beta
> 0$ to avoid the complications that can arise from the viscous term
being removed and the system becoming fully hyperbolic (these
difficulties are discussed in more detail in \cite{bonvin2001gls}). We
note that the conformation tensor is related to the polymeric
extra-stress tensor $\vec \tau$ by
\begin{equation*}
    \sig = \vec I + \frac{\operatorname{Wi}}{1-\beta}\vec \tau.
\end{equation*}
We note that in the context in this work (i.e. a lid-driven cavity), the
Weissenberg and Deborah numbers are related to each other via the aspect ratio
of the cavity \cite{pakdel1997cavity}, and are in fact equal for a cavity of
square cross-section.

We prescribe Dirichlet boundary conditions $\vec w$ on $\partial \W$ for the velocity such
that $\vec w \cdot \vec n = 0$ on $\partial \W$, where $\vec n$ is the outward
facing normal vector to the boundary, i.e. the problem is driven by tangential
boundary values for the velocity. We note that since Dirichlet boundary
conditions are prescribed on the whole boundary, to ensure that there exists a
uniquely determined pressure field we include the constraint that the pressure
has mean zero. Initial conditions for the fluid velocity and conformation tensor
are given by $\vec u(\vec x, 0) = \vec u^0(\vec x)$ and $\sig(\vec x, 0) =
\sig^0(\vec x)$ respectively.

The left hand side of Equation \eqref{eq:constitutive_law} is the \emph{upper
convected time derivative}, introduced by Oldroyd \cite{oldroyd1950formulation}
and discussed in greater detail in \S \ref{sec:lie_derivative_approximation}.
See also \cite{hinch2021oldroyd} for an overview. It gives the rate of change of
a tensor quantity in a convected coordinate system which moves and deforms with
the fluid. Such a rate is appropriate to derive constitutive laws which exhibit
material frame indifference that is, the principle that the physics of the
constitutive law should be independent of the frame of the observer.

\begin{remark}[Boundary conditions for the conformation tensor]
The problem described above is an enclosed flow. Since there is no
inflow boundary, and since the constitutive relation
\eqref{eq:constitutive_law} is hyperbolic, no boundary condition is
required for the conformation tensor.
The case of inflow boundary conditions for the velocity,
which we do not consider here, requires additional boundary
conditions for the conformation tensor. This is not a well studied
problem and another potential source of instabilities in these
flows.
\end{remark}

To complete the exposition, it is appropriate to discuss some of the
assumptions made in the derivation of the Oldroyd-B model. The
introduction of the upper convected time derivative by Oldroyd
\cite{oldroyd1950formulation} assumes that the physics of the
constitutive relation are formulated in terms of the evolution of
contravariant components of a second order tensor in the convected
frame moving and deforming with the fluid. This is a choice, justified
by experiments exhibiting rod-climbing (a phenomenon that the
Oldroyd-B model replicates), rather than physics, and other choices
may be more appropriate (see the discussion in
\cite{hinch2021oldroyd}).

The model has also been derived from kinetics using a dumbbell and spring model
in which no limit is placed upon the extension of the spring. However, if the
flow is strong (that is, if the largest eigenvalue of $\grad \vec u$ exceeds
$\tfrac{1}{2 \lambda}$, where $\lambda$ is the relaxation time
\cite{hinch2021oldroyd}), the molecule length increases without limit. It has
been shown that, as a result of this, infinite stress can occur in the interior
of a steady flow. In many regimes however, the Oldroyd-B fluid is a good
approximation of a Boger fluid. It is the simplest differential viscoelastic
model available which enjoys material frame indifference (simpler models do
exist such as the linear Maxwell model, see \cite[Chapter
2]{renardy2000mathematical} for an overview, but do not have this property). It
is also popular, with a wide range of benchmark numerical data available in the
literature. It therefore represents a prototypical model for the derivation and
testing of numerical methods, see for example 
\cite{fattal2005time,jafari2018property,lee2011stable,lukavcova2016energy,
pimenta2017stabilization} and many others. 

\begin{remark}
Taking the kinetic viewpoint, the Oldroyd-B model can be improved by
considering for example variants of the \emph{Finitely Extensible
Nonlinear Elastic} (FENE) model for the spring force in the polymer
molecules, which limits molecules to be finitely extensible and
avoiding the issue of unbounded stress growth. This does however
result in a more complicated nonlinearity in the constitutive law. For
example, the FENE-P model has constitutive law given by
\begin{equation*}
  \frac{\partial \sig}{\partial t} + (\vec u \cdot \grad) \sig 
  -
  (\grad \vec u) \sig - \sig (\grad \vec u)^T
  =
  \frac{1}{\wi}\left(\vec I -\frac{1}{1 - \frac{1}{b^2}\tr(\sig)}\sig\right),
\end{equation*}
where the additional parameter $b$ represents the maximum
extensibility of the polymer molecules. The extension of the current
work to more realistic models will be the subject of further research.
\end{remark}

\section{A Lie derivative framework for the upper convected time derivative}
\label{sec:lie_derivative_approximation}

In his seminal paper \cite{oldroyd1950formulation}, Oldroyd described the
importance of formulating a time derivative with which to construct constitutive
laws for non-Newtonian stresses that does not depend on the choice of reference
frame. Such derivatives have to take account of the moving fluid flow and its
deformation. The \emph{upper convected time derivative}, one of two derivatives introduced by Oldroyd, is one of (infinitely)
many such time derivatives. It is now well known that these rates are special
cases of Lie derivatives (see eg \cite[chapter II]{hughes1984numerical}), and
formulating the constitutive law in a way that respects this can have
structure-preserving advantages. In this section we give the concepts necessary
to formulate the constitutive law for a non-Newtonian fluid as a Lie derivative.

\subsection{Flow map and deformation gradient}
\label{sec:flowmap}

Let $\vec y$ be the flow map for the fluid motion. That is, the fluid
particle which at time $t$ has position $\vec x$ has position $\vec
y(\vec x,t;s)$ at time $s$. The flow map satisfies the following ordinary differential equation
\begin{equation*}
  \begin{split}
    \partial_s \vec y(\vec x, t; s)
    &= \vec u(\vec y(\vec x, t; s), s) 
    \\
    \vec y(\vec x,t;t) &= \vec x.
  \end{split}
\end{equation*} 
One can express the material derivative in terms of the flow map by
examining the infinitesimal derivative, that is
\begin{equation*}
  \left.\frac{\partial}{\partial s}\vec u(\vec y(\vec x, t; s), s)\right|_{s=t}
  =
  \left(\partial_t \vec + \vec u \cdot \grad \right)\vec u(x, t).
\end{equation*}
We define the velocity gradient with the convention that it has components 
\begin{equation*}
  (\grad \vec u)_{ij} 
  =
  \frac{\partial u_i}{\partial x_j},
\end{equation*}
and make note that some authors define the velocity gradient as the
transpose of the above tensor. The deformation gradient tensor
describes the deformation experienced by a fluid parcel in a
neighbourhood of $(\vec x, t)$ between times $s_1$ and $s_2$. In a Cartesian coordinate system, it has components
\begin{equation}\label{eq:def_grad_defn}
  \vec F_{ij}(\vec x, t; s_1, s_2) 
  =
  \frac{\partial \vec y_i(\vec x, t; s_2)}{\partial \vec y_j(\vec x, t; s_1)}.
\end{equation} 
It can alternatively be expressed in the following way:
\begin{equation*} 
  \vec F(\vec x, t; t, s) = \grad \vec y(\vec x, t;s).
\end{equation*}
Where no confusion can occur, we will omit the particle label and for
brevity write $\vec F(s_1, s_2)$ for $\vec F(\vec x, t; s_1, s_2)$ and
$\vec y(s)$ for $\vec y(\vec x,t;s)$.

It follows immediately from the definition that $\vec F(s, s) = \vec
I$, where $\vec I$ is the identity tensor, and that $\vec F(s_1,
s_2)=\vec F(s_2, s_1)^{-1}$. The deformation gradient is a two point
tensor, and the following lemma describes its evolution with respect
to the two faux time indices:

\begin{lemma}[Evolution of the deformation gradient along
characteristics]\label{lem:evolution_of_F}
The deformation gradient tensor satisfies the following differential equations.
  \begin{equation}\label{eq:def_grad_ode_1}
    \frac{\partial}{\partial s_1}\vec F(\vec x, t; s_1, s_2)
    =
    - \vec F(\vec x, t; s_1, s_2) \grad \vec u(\vec y(\vec x, t; s_1), s_1)
  \end{equation}
  \begin{equation}\label{eq:def_grad_ode_2}
    \frac{\partial}{\partial s_2}\vec F(\vec x, t; s_1, s_2)
    =
    \grad \vec u(\vec y(\vec x, t; s_2), s_2) \vec F(\vec x, t; s_1, s_2).
  \end{equation}
\end{lemma}

\begin{proof}
  We begin with Equation \eqref{eq:def_grad_defn} and use the chain
  rule to compute the components of $\partial_{s_2}\vec F(s_1, s_2)$:
  \begin{equation*}
    \begin{split}
    \frac{\partial}{\partial s_2}\vec F( s_1, s_2)_{ij}
    &=
    \frac{\partial}{\partial s_2}
    \left(\frac{\partial \vec y_i(s_2)}{\partial \vec y_j(s_1)}\right)
    = 
    \frac{\partial \vec u_i(\vec y(s_2),s_2)}{\partial \vec y_j(s_1)}
    =
    \frac{\partial \vec u_i(\vec y(s_2),s_2)}{\partial \vec y_k(s_2)}
    \frac{\partial \vec y_k(s_2)}{\partial \vec y_j(s_1)}\\
    &= 
    \grad \vec u(\vec y(s_2),s_2)_{ik}\vec F(s_1, s_2)_{kj},
  \end{split}
  \end{equation*}
  showing \eqref{eq:def_grad_ode_2}. The proof of
  \eqref{eq:def_grad_ode_1} follows in a similar manner. 
  \end{proof}

\begin{proposition}\label{prop:def_f_is_1}
  The determinant of the deformation gradient $\vec F$ is an invariant
  of Equations \eqref{eq:def_grad_ode_1} and
  \eqref{eq:def_grad_ode_2}. In particular, $\det\qp{\vec F(s_1, s_2)}
  = 1$ for all $s_1, s_2$.
\end{proposition} 

\begin{proof} 
  For any $s$, we have $\det\qp{\vec F(s, s)} = \det \qp{\vec I} =
  1$. Now, \eqref{eq:def_grad_ode_2} is a tensor-valued differential
  equation with coefficient matrix $\grad \vec u(\vec y(s_2),
  s_2)$. Since the fluid is incompressible, the trace of the velocity
  gradient is zero. We can therefore invoke \cite[\S IV.3, Lemma
    3.1]{hairer2006geometric} which guarantees $\det\qp{\vec F(s_1,
    s_2)}$ is invariant.
\end{proof}

\subsection{The upper convected time derivative}

For a generic sufficiently smooth tensor function $\vec \zeta$, the
upper convected derivative is defined through
\begin{equation*}
  \overset{\triangledown}{\vec \zeta}
  :=
  \frac{\partial \vec \zeta}{\partial t} + (\vec u \cdot \grad) \vec \zeta
  - (\grad \vec u)\vec \zeta - \vec \zeta(\grad \vec u)^T.
\end{equation*}
This has a relationship to the notion of a \emph{Lie derivative} of
the tensor quantity $\vec \zeta$ which we will exploit in this
work. Indeed, one may introduce the Lie derivative which follows the
fluid with deformation gradient tensor $\vec F$ by
\begin{equation*}
  (\mathcal{L}_u \vec \zeta)(\vec y(\vec x, t; s), s)
  :=
  \vec F(t, s)\frac{\partial}{\partial s}
  \left[\vec F(s,t)\vec \zeta(\vec y(s),s) \vec F(s,t)^T\right]\vec F(t,s)^T.
\end{equation*}

The upper-convected derivative can be inferred from the fundamental
properties of $\vec F$ given in \S\ref{sec:flowmap}
\cite{lee2006new,medeiros2021second}
\begin{equation*}
  \overset{\triangledown}{\vec \zeta} = \lim_{s \to t} (\mathcal{L}_u
  \vec \zeta)(\vec y(\vec x, t; s), s).
\end{equation*}
The Oldroyd-B constitutive relation for the conformation tensor can therefore be
formulated in terms of the Lie derivative resulting in an equation in one
(characteristic) time variable as follows
\begin{equation}\label{eq:Oldroyd_B_Lie}
  \overset{\triangledown}{\sig}
  =
  \lim_{s \to t} (\mathcal{L}_u \sig)(\vec y(\vec x, t; s), s).
  =
  -\frac{1}{\operatorname{Wi}} \left(\sig - \vec I\right).
\end{equation}
We see that the constitutive law in Oldroyd-B is a linear ODE along
the characteristics induced by the deformation gradient $\vec F$. This
is why Oldroyd-B is often referred to as a linear viscoelastic
model. It also offers the opportunity to utilise the a wide variety of
timestepping methods designed for problems of this kind.

\section{Discretisation of the upper convected time
derivative}\label{sec:semi_disc}

In this section we introduce a discretisation of the upper convected time derivative motivated by its formulation as a Lie derivative given in \S\ref{sec:lie_derivative_approximation}. The method will be similar to characteristic schemes, where the material
derivative is discretised along the characteristics of the flow
\cite{pironneau1982transport,rui2002second,suli1988convergence}. Adapting this
methodology to account for the deformation of the fluid provides a natural way
to preserve the positive definiteness of the conformation tensor of any given
discretisation. This idea has been explored previously in the context of several different numerical frameworks \cite{boyaval2009free,lee2011stable,medeiros2021second,petera2002new}.

We divide the time domain $[0,T]$ into a uniform partition of $N$
subintervals of length $\Delta t$ (so that $T = N \Delta t$). We write
$t^n = n \Delta t$ for $n = 0,1,...,N$, non uniform timestepping is
certainly possible however for clarity of exposition we will restrict
to uniform. We also use the notation $\vec u^n(\vec x) = \vec u(\vec
x, t^n)$. Starting from
\begin{equation}\label{eq:Lie_derivative_def_1}
  (\mathcal{L}_u \vec \zeta)(\vec y(\vec x, t; s), s)
  =
  \vec F(t, s)\frac{\partial}{\partial s}
  \left[\vec F(s,t)\vec \zeta(\vec y(s),s) \vec F(s,t)^T\right]\vec F(t,s)^T,
\end{equation}
we will utilise a finite difference approach which motivates the
following definitions:

\tikzset{pics/segment/.style args=
{angle #1 left #2 right #3}{
code={\draw[->,rotate=#1] (180:#2)--(0:#3);}}} 
\begin{figure}
  \begin{center}
    \fbox{
\begin{tikzpicture}
  \clip (-1, 1.5) rectangle (9, 4.5); 
  \def\mecurve{(0,2) ..controls +(60:6) and +(-120:6) .. (8,4)}
  \draw \mecurve
  coordinate[pos=.85] (V)
  coordinate[pos=.65] (R)
  coordinate[pos=.71] (Q)
  coordinate[pos=.2] (P);
  \foreach \t in {.71}
  \path \mecurve 
  pic[pos=\t,sloped,cyan,thick]
  {segment=angle 0 left 0 right 2.5};
  \foreach \t in {.71}
  \path \mecurve 
  pic[pos=\t,sloped,black,dashed,thick]
  {segment=angle 0 left 3.8 right 0};
  \fill (1.7,2.83) circle(2pt) node[below]{$\vec y^t_{\Delta t}(\vec x)$};
  \fill (Q) circle(2pt) node[above]{$(\vec x, t)$};
  \fill (P) circle(2pt) node[above right]{$\vec y(\vec x, t; t-\Delta t)$};
  \node[right] at (6.4,1.8) {$\vec u(\vec x, t)$};
\end{tikzpicture}
 }
    \caption{\label{fig:flow_map_illustration}Graphical illustration of the departure point $\vec y(\vec x, t; t-\Delta t)$ of $(\vec x, t)$ under the flow map, and the approximate departure point $\vec y^t_{\Delta t}(\vec x)$. The curve represents the characteristic that the particle $(\vec x, t)$ belongs to.}
  \end{center}
\end{figure}
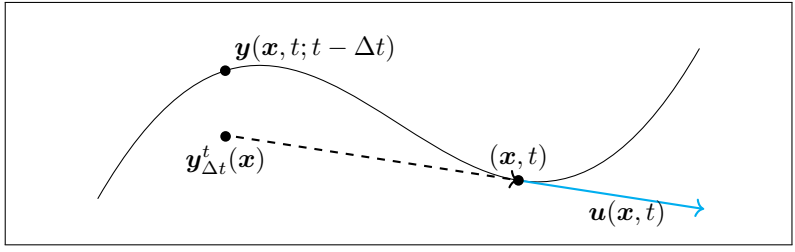

\begin{definition}[Semidiscrete upper convected derivative]\label{def:disc_uc}

  For $\vec x \in \W$, $0 < \Delta t < T$ and $t \in [\Delta t, T]$, define 
  \begin{equation}\label{eq:approx_depart_point}
    \vec y^t_{\Delta t}(\vec x) 
    =
    \vec x - \Delta t \, \vec u(\vec x, t),
  \end{equation}
  as an approximation of the departure point of $\vec x$ under the
  fluid velocity at time $t$. Further, let
  \begin{equation}\label{eq:approx_def_grad}
    \vec F^t_{\Delta t}(\vec x) 
    =
    \vec I + \Delta t \, \grad \vec u(\vec x, t),
  \end{equation}
  denote an approximation of the deformation gradient tensor between
  times $t-\Delta t$ and $t$.  We define a discrete upper-convected
  time derivative of a generic, sufficiently smooth, tensor function
  $\vec \zeta$ by the velocity field $\vec u(\vec x, t)$ as
  \begin{equation*}
    D_{(\vec u, \,\Delta t)}\vec \zeta(\vec x, t)
    =
    \frac{1}{\Delta t}
    \Big(\vec \zeta(\vec x, t) - \vec F^t_{\Delta t}(\vec x)
    \vec \zeta(\vec y^t_{\Delta t}, t - \Delta t)
    \vec F^t_{\Delta t}(\vec x)^T\Big).
  \end{equation*}
\end{definition}

\begin{remark}
  \label{rem:different_derivative_approximations}
  Different choices are available to make this approximation depending
  on the desired properties of the scheme. Definition
  \ref{def:disc_uc} results from choosing a backward Euler
  discretisation, whereby the right hand side of Equation
  \eqref{eq:Lie_derivative_def_1} is approximated by
  \begin{equation*}
    \vec F(t, s)\frac{1}{\Delta t}\Big\{\vec F(s,t)\vec \zeta(\vec y(s),s)
    \vec F(s,t)^T
    - 
    \vec F(s-\Delta t,t)\vec \zeta(\vec y(s-\Delta t),s-\Delta t)
    \vec F(s-\Delta t,t)^T\Big\}\vec F(t, s)^T,
  \end{equation*}
  which after evaluating at $s=t$ becomes
  \begin{equation}\label{eq:lie_derivative_approx}
    \begin{split}
    \frac{1}{\Delta t}\left\{\vec \zeta(\vec x, t)
    - 
    \vec F(t-\Delta t,t)\vec \zeta(\vec y(t-\Delta t),t-\Delta t) 
    \vec F(t-\Delta t,t)^T\right\}.
    \end{split}
  \end{equation}
  It is not always practical (or possible) to exactly evaluate the deformation
  gradient and flow map, so they are replaced with approximations $\vec
  F^t_{\Delta t}$ and $\vec y^t_{\Delta t}$ respectively in the definition of
  $D_{(\vec u, \,\Delta t)}$.
  
  The structure associated to (\ref{eq:Oldroyd_B_Lie}) is dissipative
  in nature hence a BDF discretisation is quite natural. A detailed
  discussion of approximations of the upper convected time derivative
  in the manner described above is given in \cite{medeiros2021second},
  where the authors discuss first and second order approximations of
  BDF type.
\end{remark}

We now turn our attention to well-posedness and properties of the
scheme. We first show in Lemma \ref{lem:well-defined} that $D_{(\vec
  u, \,\Delta t)}$ is well-defined as long as $\Delta t$ is chosen to
be sufficiently small. For a velocity $\vec u$ satisfying a
homogeneous Dirichlet boundary condition, this result is proved in
\cite{rui2002second}. We provide an extension of that result
appropriate for nonzero tangential velocity boundary conditions such
as the lid-driven cavity.

\begin{lemma}\label{lem:well-defined} Let $\W$ be the unit square or
  the unit cube. Assume that $\vec u \cdot \vec n = 0$ on $\partial
  \W$, and that $\Delta t$ is chosen so that $\Delta t \Norm{\vec
    u}_{C^0([0,T];\sob{1}{\infty}(\W))}<1$.  Then for any $\vec x \in
  \W$, and any $t \in [\Delta t, T]$, we have $\vec y^t_{\Delta
    t}(\vec x) \in \W$.
\end{lemma}
\begin{proof}
  Let $\vec x \in \W$. Then, by (\ref{eq:approx_depart_point})
  \begin{equation*}
    |\vec x - \vec y^t_{\Delta t}(\vec x)| = \Delta t |\vec u(\vec x, t)|.
  \end{equation*}
  The boundary of $\W$ has the form $\partial \W = \cup_{i=1}
  \Gamma_i$ where $\Gamma_i$ is a line segment when $d=2$ and a plane
  segment when $d=3$.  Let $\vec x \in \W$ and for some $i$, denote by
  $\vec x_{\Gamma_i}$ the unique point on $\Gamma_i$ that satisfies
  $|\vec x- \vec x_{\Gamma_i}| = \text{dist}(\vec x, \Gamma_i)$. As
  $\W$ is convex this point clearly lies in the interior of
  $\Gamma_i$, so that the unit outward normal to the boundary at $\vec
  x_{\Gamma_i}$ is well-defined, and the line segment joining $\vec x$
  and $\vec x_{\Gamma_i}$ is perpendicular to $\Gamma_i$. We consider
  the component of the velocity in the direction of $\vec n(\vec
  x_{\Gamma_i})$. Since $\vec u(\vec x_{\Gamma_i},t)\cdot \vec n(\vec
  x_{\Gamma_i}) = 0$,
  \begin{equation*}
    \begin{split}
    \vec u(\vec x, t)\cdot \vec n(\vec x_{\Gamma_i})
    &=
    \int_0^1\frac{\partial}{\partial s}\vec u(\vec x_{\Gamma_i} 
    + s(\vec x - \vec x_{\Gamma_i}),t)\cdot \vec n(\vec x_{\Gamma_i}) \diff s,\\
    &= (\vec x - \vec x_{\Gamma_i})\cdot 
    \int_0^1 \grad \vec u(\vec x_{\Gamma_i} 
    + 
    s(\vec x - \vec x_{\Gamma_i}),t)\cdot \vec n( \vec x_{\Gamma_i}) \diff s \\
    &\leq \text{dist}(\vec x, \Gamma_i)
    \Norm{\vec u}_{C^0([0,T];\sob{1}{\infty}(\W))},
    \end{split}
  \end{equation*} 
  where $\vec n(\vec x_{\Gamma_i})$ is the unit normal to $\partial
  \W$ at $\vec x_{\Gamma_i}$. Again, using
  (\ref{eq:approx_depart_point}) we have shown that
 \begin{equation*}
  |(\vec x - \vec y^t_{\Delta t}(\vec x)) \cdot \vec n(\vec x_{\Gamma_i})| 
  \leq
   \text{dist}(\vec x, \Gamma_i)\Delta t
   \Norm{\vec u}_{C^0([0,T];\sob{1}{\infty}(\W))}
   \leq 
   \text{dist}(\vec x,\Gamma_i ),
 \end{equation*}
 that is, the velocity at $\vec x$ in the direction of $\vec n(\vec
 x_{\Gamma_i})$ is not large enough for the line segment joining $\vec x$ and
 $\vec y^t_{\Delta t}(\vec x)$ to cross $\Gamma_i$. Since
 the choice of $i$ was arbitrary, we must have that $\vec y^t_{\Delta t}(\vec x)
 \in \W$.
\end{proof}

In the following lemmata, we discuss the
approximation properties of $\vec y^t_{\Delta t}$, $\vec F^t_{\Delta
  t}$ and finally $D_{(\vec u, \,\Delta t)}$. We will utilise certain
results which were proved in \cite{medeiros2021second}. The first,
Lemma \ref{lem:characteristics}, gives a bound on the error arising
from approximating the departure point using
\eqref{eq:approx_depart_point} (see Figure \ref{fig:flow_map_illustration} for an illustration), while Lemma
\ref{lem:function_at_point} bounds the error in evaluating a tensor
function at this approximate departure point as opposed to the true
point. The error in approximating the deformation gradient is bounded
in Lemma \ref{lem:disc_def_grad}, and finally a bound showing the
consistency and approximation order of the discrete upper convected
derivative $D_{(\vec u, \,\Delta t)}$ is given in Theorem
\ref{thm:upper_convected}.

\begin{lemma}[Approximation of characteristics {\cite[Lemma 2]{medeiros2021second}}]
  \label{lem:characteristics}
 Suppose that $\vec u \cdot \vec n = 0$ on $\partial \W$ and
 that\\ $\Delta t \Norm{\vec
   u}_{C^0([0,T];\sob{1}{\infty}(\W))}<1$. Then for any $\vec x \in
 \W$ and $t \geq \Delta t$,

\begin{equation}\label{eq:characteristic_result}
  \vec y(\vec x, t; t-\Delta t) 
  -
  \vec y^t_{\Delta t}(\vec x)
  = 
  \frac{\Delta t^2}{2} \frac{D\vec u}{Dt}
  +
  \mathcal{O}(\Delta t^3).
\end{equation}
\end{lemma}

\begin{lemma}\label{lem:function_at_point} In addition to the hypotheses of
  Lemma \ref{lem:characteristics}, let $\vec \zeta$ be an appropriately
  differentiable tensor function. Then for any $\vec x \in \W$
  and $t \geq \Delta t$, 
\begin{equation*}
  \vec \zeta(\vec y(\vec x, t; t-\Delta t), t - \Delta t)
  -
  \vec \zeta(\vec y^t_{\Delta t}(\vec x),t-\Delta t) 
  =
  \frac{\Delta t^2}{2}\left(\frac{D \vec u}{Dt} \cdot \grad\right) \vec \zeta
  +
  \mathcal{O}(\Delta t^3).
\end{equation*}
\end{lemma}

\begin{lemma}[Approximation of deformation gradient {\cite[Lemma 1]{medeiros2021second}}]\label{lem:disc_def_grad}
  Suppose that the hypotheses of Lemma \ref{lem:characteristics} are satisfied.
  Then for any $\vec x \in \W$ and $t \geq \Delta t$, 
  \begin{equation*}
    \vec F(t - \Delta t, t) 
    -
    \vec F^t_{\Delta t}(\vec x)
    =
    \Delta t^2 \left[(\grad \vec u)^2 - \frac{D (\grad\vec u)}{Dt}\right] 
    + 
    \mathcal{O}(\Delta t^3).
  \end{equation*}
\end{lemma}

\begin{remark}
  The bounds of Lemma \ref{lem:characteristics} and Lemma
  \ref{lem:disc_def_grad} require spatial and temporal regularity of $\vec u$
  due for example to the appearance of the material derivative of $\vec u$ on
  the right hand side of Equation \eqref{eq:characteristic_result}. In addition,
  the result of Lemma \ref{lem:function_at_point} requires spatial regularity of
  $\vec \zeta$.
\end{remark}

We finally arrive at the approximation of the upper convected derivative (see {\cite[Remark 7]{medeiros2021second}}).
\begin{Theorem}\label{thm:upper_convected} Suppose that the hypotheses of Lemma
  \ref{lem:characteristics} are satisfied, and that $\vec \zeta$ is an
  appropriately differentiable tensor function. Then for any $\vec x \in \W$ and
  $t \geq \Delta t$,
  \begin{equation*}
  \overset{\triangledown}{\vec \zeta}(\vec x, t)
  =
  D_{(\vec u, \,\Delta t)}\vec \zeta(\vec x, t)
  +
  \mathcal{O}(\Delta t).
\end{equation*}
\end{Theorem}

Theorem \ref{thm:upper_convected} may be interpreted as an expansion
around $\vec x, t$ in that if follows directly that we have, at $\vec
y^t_{\Delta t}(\vec x)$, 
\begin{equation*}
  \vec F^t_{\Delta t}(\vec x)
  \vec \zeta(\vec y^t_{\Delta t}, t - \Delta t) 
  \vec F^t_{\Delta t}(\vec x)^T
  =
  \vec \zeta(\vec x, t)
  - 
  \Delta t \overset{\triangledown}{\vec \zeta}(\vec x, t)
  + 
  \mathcal{O}(\Delta t^2).
\end{equation*}
We are now in a position to state the semidiscrete scheme for the
constitutive law (\ref{eq:constitutive_law}).

\subsection{Semidiscretisation of the Oldroyd-B constitutive law}

Suppose that $\vec u(\vec x,t)$ is given for $\vec x \in \W$ and $t
\in [0,T]$. Suppose also that an initial condition $\sig^0(\vec x)$
is given. Then for $n \geq 1$ we define $\vec \Sigma^n$ by

\begin{equation}\label{eq:semidiscrete_constitutive_law}
  \begin{split}
    D_{(\vec u, \,\Delta t)}\vec \Sigma^n(\vec x)
    &=
    -
    \frac{1}{\wi}\left(\vec \Sigma^n(\vec x)-\vec I\right)
    \\
    \vec \Sigma^0(\vec x) &= \sig^0(\vec x).
  \end{split}
\end{equation}

\begin{proposition}[Consistency]
  The semidiscrete scheme given in
  \eqref{eq:semidiscrete_constitutive_law} is a consistent, first
  order approximation.
\end{proposition}
\begin{proof}
  By Theorem \ref{thm:upper_convected}, we have the following representation:
  \begin{equation*}
    \sig(\vec x, t) 
    = 
    \Delta t \, \overset{\triangledown}{\sig}(\vec x, t)
    +
    \vec F^t_{\Delta t}(\vec x)
    \sig(\vec y^t_{\Delta t}(\vec x), t - \Delta t) 
    \vec F^t_{\Delta t}(\vec x)^T
    + 
    \mathcal{O}(\Delta t^2).
  \end{equation*}
  Then, using Equation \eqref{eq:constitutive_law}, we obtain
  \begin{equation}\label{eq:trunc_1}
    \left(1 + \frac{\Delta t}{\wi}\right)
    \sig(\vec x, t)
    =
    \vec F^t_{\Delta t}(\vec x)
    \sig(\vec y^t_{\Delta t}(\vec x), t - \Delta t)
    \vec F^t_{\Delta t}(\vec x)^T 
    - 
    \frac{\Delta t}{\wi} \vec I
    +
    \mathcal{O}(\Delta t^2).
  \end{equation}
  Assuming that $\sig(\vec y^t_{\Delta t}(\vec x), t - \Delta t)$ is
  given, one step of the semidiscrete scheme gives that $\vec
  \Sigma^n$ satisfies
  \begin{equation}\label{eq:trunc_2}
    \left(1 + \frac{\Delta t}{\wi}\right)\vec \Sigma^n(\vec x)
    =
    \vec F^t_{\Delta t}(\vec x)
    \sig(\vec y^t_{\Delta t}(\vec x), t - \Delta t)
    \vec F^t_{\Delta t}(\vec x)^T 
    - 
    \frac{\Delta t}{\wi} \vec I.
  \end{equation}
  Finally, noting that for any $\Delta t > 0$, we have $0<\tfrac{\wi}{\wi + \Delta t}<1$, so that subtracting \eqref{eq:trunc_2} from \eqref{eq:trunc_1} gives
  \begin{equation*}
    \sig(\vec x, t) - \vec \Sigma^n(\vec x) 
    = 
    \mathcal{O}(\Delta t^2),
  \end{equation*}
  and therefore the semidiscrete scheme has local truncation error proportional to $\Delta t^2$.
\end{proof}

\begin{proposition}[The positive definiteness of the conformation tensor is 
  preserved]
  \label{prop:pos_def_semi} 
\noindent Suppose that

\noindent $\vec \Sigma^{n-1}$ is given, and is positive
definite, and assume that $\vec u$ is given. Then there exists $\Delta t > 0$
such that $\vec \Sigma^n$ is positive definite.
\end{proposition}

\begin{proof}
We can write
\eqref{eq:semidiscrete_constitutive_law} as 
  \begin{equation}\label{eq:update_step}
    \left(1 + \frac{\Delta t}{\wi}\right)\vec \Sigma^n(\vec x)
    =
    \vec F^t_{\Delta t}(\vec x)
    \vec \Sigma^{n-1}(\vec y^t_{\Delta t}(\vec x))
    \vec F^t_{\Delta t}(\vec x)^T
    +
    \frac{\Delta t}{\wi} \vec I.
  \end{equation}
  Thus, as long as $\Delta t$ is chosen small enough, a sufficient
  condition would be $\Delta t \leq \tfrac 1 2 \Norm{\vec u}_{1,
    \infty}$ for example, $\vec F^t_{\Delta t}$ is invertible by the
  Gershgorin circle theorem \cite[Thm 7.2.1]{GolubVan-Loan:2013}. Let
  $\vec 0 \neq \vec \xi \in \reals^d$. Then since $\vec F^t_{\Delta
    t}$ is invertible, the product $(\vec F^t_{\Delta t})^T \vec \xi$
  is a nonzero vector in $\reals^d$, so that by hypothesis
  \begin{equation*}
    0 < \vec \xi^T \vec F^t_{\Delta t} 
    \vec \Sigma^{n-1}(\vec y^t_{\Delta t}(\vec x)) 
    (\vec F^t_{\Delta t})^T \vec \xi.
  \end{equation*}
  Then from \eqref{eq:update_step}, $\vec \Sigma^n(\vec x)$ is a sum of two positive definite terms.
\end{proof}

\section{Spatial discretisation}\label{sec:numerical_schemes}

In this section we present numerical methods for the Oldroyd-B fluid in the case
$\operatorname{Re} = 0$. The problem is solved in a decoupled manner, in which
at each time step a Stokes problem is solved for the fluid velocity, which is
then used as a forcing in the time-dependent constitutive relation.  This avoids
the need to solve a very large monolithic system and therefore keeps
computational costs down. We first give a finite element method to solve the
Stokes problem and discuss well-posedness. We then present the fully discrete
constitutive law and discuss the solution procedure for the full problem.

\subsection{Finite element discretisation for the Stokes problem}

We consider $\T{}$ to be a conforming triangulation of $\W$, namely,
$\T{}$ is a finite family of sets such that
\begin{enumerate}
\item $K\in\T{}$ implies $K$ is an open box (i.e. a quadrilateral or
hexahedron),
\item for any $K,J\in\T{}$ we have that $\closure K\cap\closure J$ is a full
  lower-dimensional box (i.e., it is either $\emptyset$, a vertex,
  an edge (or face) or the whole of $\closure K$ and $\closure J$) of both
  $\closure K$ and $\closure J$ and
\item $\bigcup_{K\in\T{}}\closure K=\closure\W$.
\end{enumerate}
Further, we define $h: \W\to \reals$ to be the {piecewise constant}
\emph{meshsize function} of $\T{}$ given by
\begin{equation*}
  h(\vec{x}):=\max_{\closure K\ni \vec{x}}\text{diam} (K).
\end{equation*}

With $\mathbb Q_1(K)$ the space of bilinear polynomials over a quadrilateral (or
trilinear over a hexahedron), we introduce the
\emph{finite element space}
\begin{equation*}
  \mathbb{Q}_1 := \{\phi \in \sobh1(\W) : \phi|_K \in\mathbb Q_1(K)\},
\end{equation*}
to be the usual space of continuous piecewise bilinear functions.  Let
the vertices of the triangulation $\T{}$ be denoted $\vec x_i$.  For a
vertex $\vec x_i$, we define
\begin{equation*}
  \hat{\vec x}_i 
  :=
  \{K \in \T{} : \vec x_i \in \bar{K}\}.
\end{equation*} 

To facilitate strong imposition of boundary conditions for the velocity, we introduce finite element spaces 
\begin{equation*}
  \fes_{\vec w} := \{\vec v \in\mathbb Q_1^d : \vec v \mid_{\partial \W} = \Pi_1 \vec w\}
\end{equation*} 
\begin{equation*}
  \fes_{\vec 0} := \{\vec v \in\mathbb Q_1^d : \vec v \mid_{\partial \W} = \vec 0\}.
\end{equation*} 

To solve the Stokes problem, we employ the stabilised equal-order
approximation introduced in \cite{dohrmann2004stabilized}. The
stabilisation makes use of the $\leb 2$ projection operator $\Pi_0:
\mathbb{Q}_1 \to \mathbb{P}_0$ that maps the finite element space to a
space of discontinuous functions which are constant on each
element. The construction of the stabilisation term is therefore local
to elements. In addition, this scheme performs well on graded meshes,
which are a useful tool due to the sharp boundary layers that occur in
viscoelastic fluid flows (see \S\ref{sec:mesh_design}). A more
detailed discussion and analysis for this scheme can be found in
\cite{bochev2006stabilization,dohrmann2004stabilized,elman2014finite}. Let
$\vec \Sigma_h^{0}$ be a given interpolant of $\vec \Sigma^0$. Then
the discrete problem for the velocity and pressure is to find $\vec
u_h \in \fes_{\vec w}, p_h \in \mathbb{Q}_1$ such that
\begin{equation}\label{eq:velocity_stabilised_galerkin}
    2 \beta \int_{\Omega} \symgrad{\vec u_h} : \symgrad{\vec v_h}
    - 
    \int_{\Omega}p_h \operatorname{div}\vec v_h 
    - 
    \int_{\Omega}q_h \operatorname{div}\vec u_h
    -
    \int_{\Omega}(p_h - \Pi_0 p_h)(q_h - \Pi_0 q_h)
    =    
    \frac{\beta - 1}{\operatorname{Wi}} 
    \int_{\Omega}{\vec \Sigma_h^{0} : \symgrad{\vec v_h}}
\end{equation}
for all $\vec v_h \in \fes_{\vec 0}, q_h \in \mathbb{Q}_1$. 
\begin{proposition}[Inf-sup stability {\cite{bochev2006stabilization}}]
  Let
  \begin{equation}
    \label{eq:disstokes}
    B_h\qp{(\vec u_h, p_h),(\vec v_h, q_h)}
    :=
    2 \beta \int_{\Omega} \symgrad{\vec u_h} : \symgrad{\vec v_h}
    - 
    \int_{\Omega}p_h \operatorname{div}\vec v_h 
    - 
    \int_{\Omega}q_h \operatorname{div}\vec u_h
    -
    \int_{\Omega}(p_h - \Pi_0 p_h)(q_h - \Pi_0 q_h).
  \end{equation}
  Then, there exists a $C>0$ independent of $h$ such that
  \begin{equation}
    \label{eq:infsup}
    \sup_{\vec v_h \in \mathbb{Q}_1^d, q_h \in \mathbb{Q}_1 }
    \frac{B_h\qp{(\vec u_h, p_h),(\vec v_h, q_h)}}
         {\Norm{\vec v_h}_{\sobh1(\W)} + \Norm{q_h}_{\leb{2}(\W)}}
         \geq
         C
         \qp{\Norm{\vec u_h}_{\sobh1(\W)} + \Norm{p_h}_{\leb{2}(\W)}}       
         \Foreach \vec u_h \in \mathbb{Q}_1^d, p_h \in \mathbb{Q}_1.
  \end{equation}
\end{proposition}
\noindent
The inf-sup condition guarantees that problem \eqref{eq:velocity_stabilised_galerkin} is well-posed.

\subsection{A finite difference approximation of the upper convected time derivative}

The discretisation of the constitutive law requires several definitions relating
to approximations of characteristics and deformation gradients, which we now
give. We write $\vec u_h^n$ for the discrete solution at time
$t^n$.

\begin{remark}
  If we assume $\vec u\in \sob1\infty(\W)\cap \sobh2(\W)$ then standard
  arguments can be used to infer the discrete velocity $\vec u_h^n \in
  \sob{1}{\infty}(\W)$. Also since $\vec u_h^n \cdot \vec n = 0$ on
  $\partial \W$ we have, in light of Lemma \ref{lem:well-defined}, the
  approximation \eqref{eq:fully_discrete_ucd} is also well-defined.
\end{remark}

\begin{definition}
Let $\vec u_h^n \in \mathbb{Q}_1^d$ be given. For any $\vec x \in \W$ and $n
\geq 1$, define the \emph{discrete departure point} of $\vec x$ according to the
velocity $\vec u_h^n$ as 
\begin{equation}\label{eq:discrete_departure_point}
  \vec y_{h,\Delta t}^{n}(\vec x)
  =
  \vec x - \Delta t \vec u_h^n(\vec x).
\end{equation}
We also define the discrete deformation gradient between time $(n-1)\Delta t$
and $n \Delta t$ as 
\begin{equation*}
  \vec F_{h,\Delta t}^n  
  =
  \vec I + \Delta t \grad \vec u_h^n(\vec x).
\end{equation*}
\end{definition}

\begin{definition}[Fully discrete approximate upper convected time derivative]
  Let $\vec u_h^n$ be given, and suppose that $\vec \zeta^{n-1},\vec \zeta^{n}$ are tensor
  valued functions representing the the value of a time dependent tensor field at
  the time levels $n-1,n$. Then the fully discrete approximation to the upper convected time derivative of $\vec \zeta$ at time level $n$ is given by
  \begin{equation}\label{eq:fully_discrete_ucd}
    D^h_{(\vec u_h^n,\, \Delta t)}\vec \zeta^n(\vec x)
    :=
    \frac{1}{\Delta t}\Big(\vec \zeta^n(\vec x) - \vec F_{h,\Delta t}^n(\vec x) \vec \zeta^{n-1}(\vec y^n_{h, \Delta t})\vec F_{h,\Delta t}^n(\vec x)^T \Big).
  \end{equation}
\end{definition}

\subsection{A fully discrete scheme for the Stokes flow of an Oldroyd-B fluid}

We are now in a position to define the full finite difference discretisation of
the constitutive law. Values of the discrete conformation tensor are stored at
the vertices $\vec x_i$. At any time level $k$, these values are sufficient to
uniquely define a function $\vec \Sigma^k_h \in \mathbb{Q}_1^{d(d+1) \slash 2}$
by bilinear (or trilinear) interpolation. This representation is used to
evaluate the conformation tensor at points which are not vertices, and is the
function on the right hand side of the discrete Stokes equation
\eqref{eq:velocity_stabilised_galerkin}. By a slight abuse of notation, we will write $\vec \Sigma_h^k$ for both the set of point values at vertices, and the function in $\mathbb{Q}_1^{d(d+1) \slash 2}$ obtained by interpolation using these values.

We note that the approximation of the deformation gradient needed to evaluate
\eqref{eq:fully_discrete_ucd} requires the value of the discrete velocity
gradient, which we wish to use at the vertices of the mesh. Since the discrete
velocity lies in $\mathbb{Q}_1^d$, the point values of its gradient are not well
defined. For a vertex $\vec x_i$ of the mesh and a function $\vec v_h \in
\mathbb{Q}_1^d$, in place of the velocity gradient, we therefore use the local average over $\hat{\vec x}_i$:
\begin{equation*}
  \grad \vec v_h (\vec x_i) 
  =
  \frac{1}{|\hat{\vec x}_i|}\sum_{K \subseteq  \hat{\vec x}_i}
  \int_{K} \grad \vec v_h.
\end{equation*}

The fully discrete problem is given as follows. Let $\vec \Sigma^0_h$ be an interpolant of the initial condition. Then for $1 \leq n \leq N$, find $\vec u^n_h \in \fes_{\vec w}$, $p_h^n \in \mathbb{Q}_1$ such that  

\begin{equation}\label{eq:finite_element_update}
  B_h\qp{(\vec u^n_h, p^n_h),(\vec v_h, q_h)}
  =
  \frac{\beta - 1}{\operatorname{Wi}} 
  \int_{\Omega}{\vec \Sigma^{n-1}_h : \symgrad{\vec v_h}}
\end{equation}
for all $\vec v_h \in \fes_{\vec 0}, q_h \in \mathbb{Q}_1$, then $\vec \Sigma_h^n$ is given by

\begin{equation}\label{eq:finite_difference_update}
  D^h_{(\vec u_h^n,\, \Delta t)}\vec \Sigma_h^n(\vec x_i)
  +
  \frac{1}{\wi}(\vec \Sigma^n_h(\vec x_i) - \vec I)
  =
  0.
\end{equation}

At each mesh point, \eqref{eq:finite_difference_update} is a coupled system of
ordinary differential equations forced by the discrete fluid velocity via the
deformation gradient, and approximated in an explicit manner, and therefore the
nodal values may be updated without the need to solve a linear system.

\begin{remark}[Choice of spatial discretisation points for finite difference
  scheme] We note that since all derivatives appearing in the constitutive
  relation are approximated via \eqref{eq:lie_derivative_approx}, we are free to
  choose any set of points at which to approximate the conformation tensor, with
  the proviso that care must be taken to adequately resolve the velocity
  gradient, as this is what forces the constitutive law. The choice of these
  points so that they correspond with the degrees of freedom of a piecewise
  bilinear finite element space is a natural one, as it is then clear how to
  include the discrete conformation tensor in the momentum equation. Other possible choices are discussed in \S\ref{sec:discuss}.
\end{remark}

\begin{proposition}\label{prop:pos_def_disc} 
  If the initial condition $\sig^0$ is positive definite, and bilinear interpolation is used to obtain $\vec \Sigma_h^0$ and to evaluate $\vec \Sigma_h^k$ at points which are not vertices of the mesh, then at each time step there exists $\Delta t > 0$ such that the discrete conformation tensor remains positive definite.
\end{proposition}

\begin{proof}
  By assumption, all nodal values of $\vec \Sigma_h^{n-1}$ are positive
  definite. Then since interpolated values are weighted sums of the nodal values
  with non-negative weights, the interpolant is also positive definite. By the
  same argument as in Proposition \ref{prop:pos_def_semi}, if we choose $\Delta
  t \leq \tfrac 1 2 \Norm{\vec u_h^n}_{1, \infty}$, and if $\vec \Sigma_h^{n-1}$
  is positive definite, then $\vec \Sigma_h^n$ computed from the scheme
  \eqref{eq:finite_difference_update} will also be positive definite.
\end{proof}

\section{Numerical Experiments}\label{sec:numerics}

In this section, we present numerical results obtained by solving \eqref{eq:finite_element_update}-\eqref{eq:finite_difference_update} for a variety of model and discretisation parameters. All
simulations presented here were conducted using \texttt{deal.II}, an open source
C\texttt{++} software library providing tools for finite element computations
\cite{dealii2019design}. 

\subsection{Description of the test case}

The lid-driven cavity consists of a fluid-filled, impenetrable box in which the
flow is driven by a moving lid. The flow of a Newtonian fluid in a lid-driven
cavity is by now a very well studied problem. It is known that the flow
characteristics, see Figure \ref{fig:lidcartoon}, depend upon the Reynolds
number. Studies from the numerical literature largely agree with one another
qualitatively and quantitatively for moderate Reynolds numbers. For larger
Reynolds numbers, numerical solutions become more challenging, and differences
appear in computational results. Detailed numerical studies have provided
benchmarks for Reynolds numbers in the tens of thousands
\cite{bruneau20062d,ghia1982high}.

In experimental studies, for small Reynolds number, flow in the cavity
reaches a steady state, and, in a three dimensional cavity, remains
approximately two dimensional. For the setup in Figure
\ref{fig:lidcartoon}, i.e., with lid moving left to right, as the
Reynolds number increases the centre of the main vortex moves to the
right. Above a critical value of the Reynolds number, the flow
transitions to a time-periodic state with variation in the neutral
direction \cite{aidun1991global}. Even for inertialess non-Newtonian
flows, the symmetry of the flow is broken (see Figure
\ref{fig:lidcartoon}, right) and moves towards the upper left corner
with increasing Weissenberg number \cite{pakdel1997cavity}.

\begin{figure}
  \begin{center}
    \begin{tikzpicture}[scale=1,>=Stealth]
\draw[thick] (0,0) rectangle (4,4);
\draw[thick,->] (0,4) -- (1.5,4);
\draw[thick,->] (2.5,4) -- (4,4);
\node at (2,4.3) {Moving Lid};
\coordinate (center) at (2,2);
\draw[->] ([shift=(0:1cm)]center) arc (0:-270:1cm);
\draw[->] ([shift=(0:1.5cm)]center) arc (0:-270:1.5cm);
\node at (2,0.2) {Clockwise Circular Flow};
\end{tikzpicture}
    \includegraphics[width=0.32\textwidth,trim=0 2.2cm 0 0]{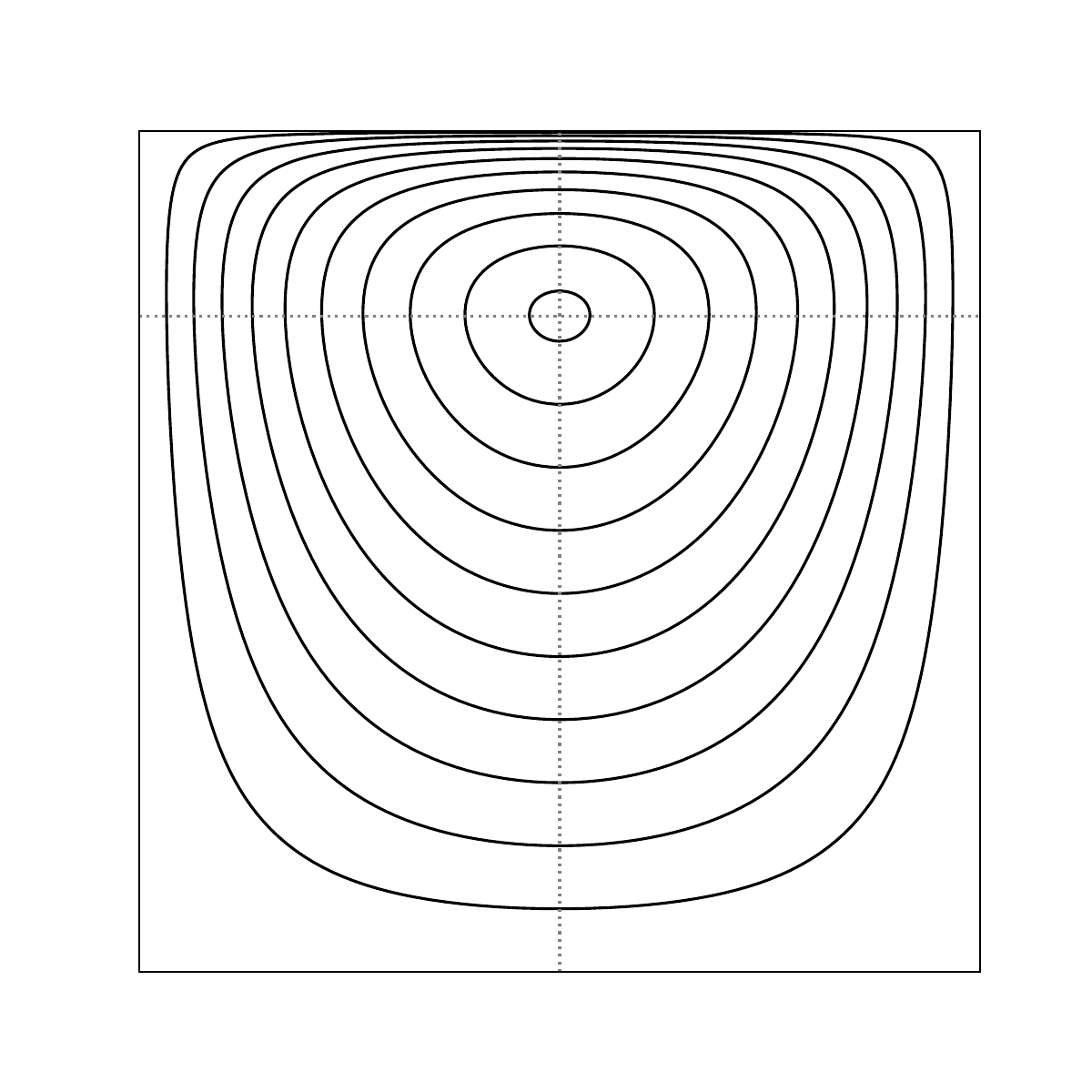}
    \includegraphics[width=0.32\textwidth,trim=0 2.2cm 0 0]{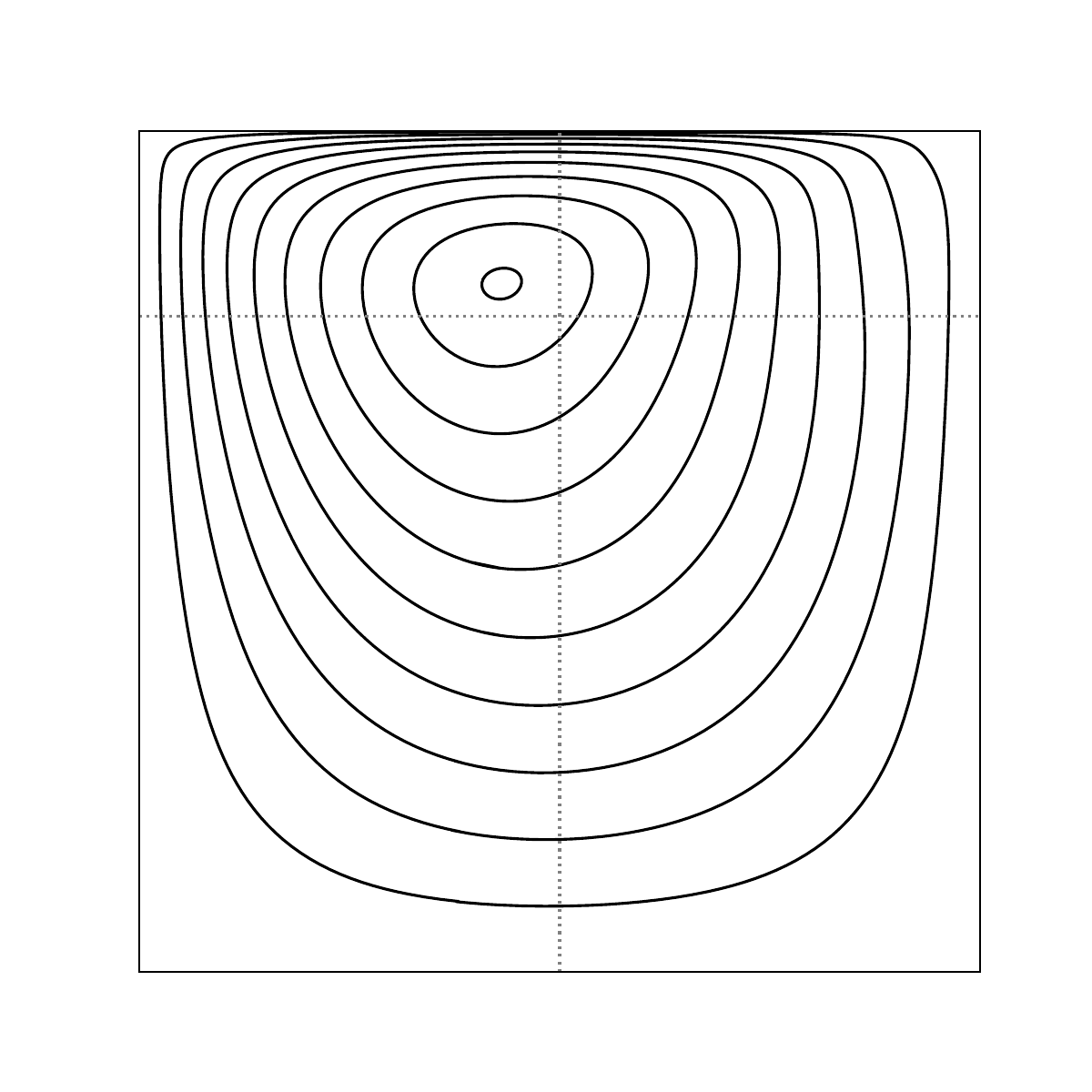}
  \end{center}
  \caption{An illustration of a lid driven cavity setup (left). Streamlines of a Newtonian Stokes flow (middle) with left-right symmetry in the velocity magnitude. Introducing a polymeric stress breaks this symmetry, with the centre of circulation moving up and to the left with increasing $\wi$ (right).
    \label{fig:lidcartoon}}
\end{figure}
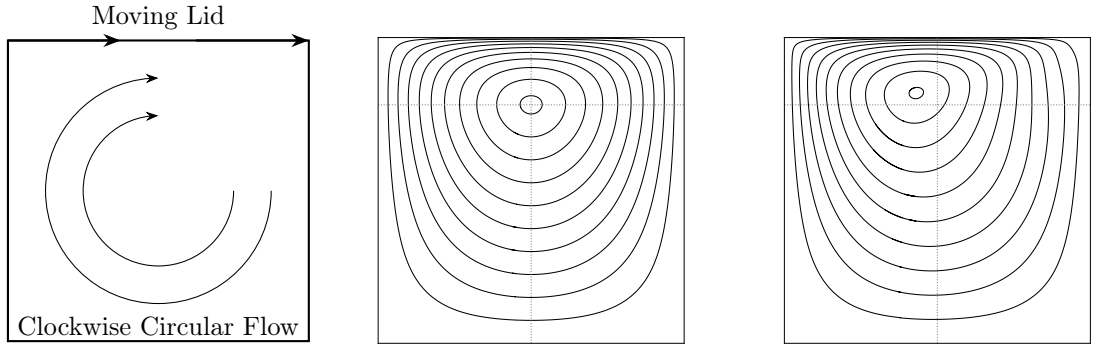
In this work, the regularised lid-driven cavity problem was solved,
for which $\W = [0,1] \times [0,1]$, $\vec u = (0,0)^T$ on the lateral and lower
boundaries, and on the lid (i.e. $\{(x_1,x_2) : x_2 = 1\}$), the velocity
boundary condition is
\begin{equation*}
  \vec u_{lid}(\vec x, t)
  =
  8\left(1 + \tanh\left(8\left(t - \tfrac 1 2\right)\right)
  x_1^2(1-x_1^2),0\right)^T.
\end{equation*}
\begin{remark}
  
For this problem, a Weissenberg number can be defined as
\begin{equation*}
  \wi 
  =
  \frac{\lambda U}{H},
\end{equation*}
where $\lambda$ is a characteristic relaxation time of the fluid, $U$ is a characteristic velocity scale and $H$ is the height of the cavity. Thus, $U\slash H$ is a measure of the shear rate \cite{pakdel1997cavity}. For the setup described above, $U=1$, $H=1$.
\end{remark}

\subsection{Mesh design}\label{sec:mesh_design}

Mesh design and selection of discretisation parameters appears to be of crucial
importance for the cavity flow of an Oldroyd-B fluid, and indeed very few
studies have reported satisfactory numerical results for the lid-driven cavity
using a uniform mesh, with the notable exception of schemes which utilise the
log-conformation tensor formulation. In addition, results for some schemes have
been shown to be very sensitive to mesh size, for example in
\cite{boyaval2010lid}, dissipation properties, positive definiteness of the
conformation tensor and existence of steady state change qualitatively upon
varying the mesh size. 

For the cavity problem, graded meshes of various types are a common choice.
Meshes constructed in this manner are able to provide very fine resolution near
the boundary (particularly the lid), where it is most needed, in an efficient
manner. The situation may be compared to singularly perturbed
convection-diffusion problems, which exhibit boundary layers that put severe
limitations on the mesh size. To make computations tractable, layer-adapted
meshes have been extensively researched in this community (see
\cite{linss2003layer} for a review). A particularly effective example is the
Bakhvalov-type mesh, where upper bounds on the derivatives of the solution near
the boundary are used to design mesh grading functions that are sufficient to
resolve it.

In this work, we will use a mixture of the following. One construction consists
of decreasing element size towards the boundary with a constant contraction
ratio $\gamma$, i.e. $h_{i+1} = \gamma h_i$ where the index $i$ is such that
$i+1$ is the next element in the direction of the boundary
\cite{comminal2015robust,sousa2016lid}. Another is to decrease element size by a
constant amount $c$ so that $h_{i+1} = h_i - c$ \cite{pan2009simulation}. For
comparison, the latter is equivalent to decreasing the contraction ratio, so
that elements size decreases very rapidly near the boundary. Finally, in
\cite{habla2014numerical}, mixtures of uniform grid spacing and refined areas
are used. 

\begin{figure}
  \centering
  \includegraphics[width=0.4\textwidth]{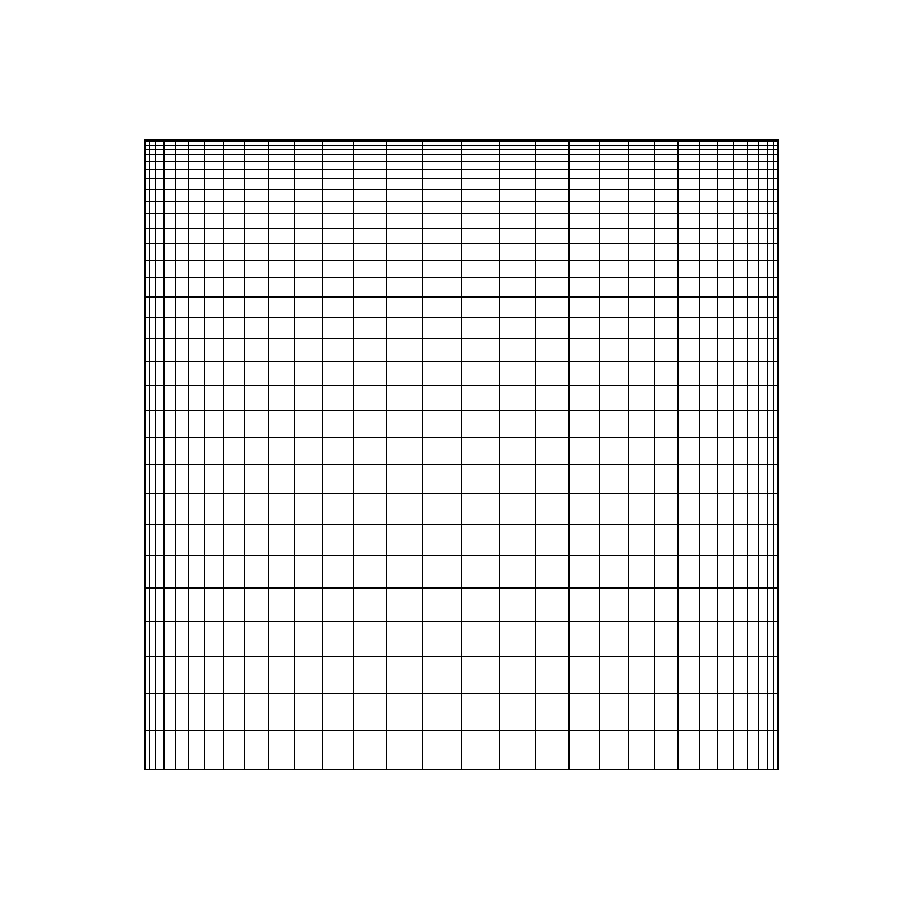}
  \includegraphics[width=0.4\textwidth]{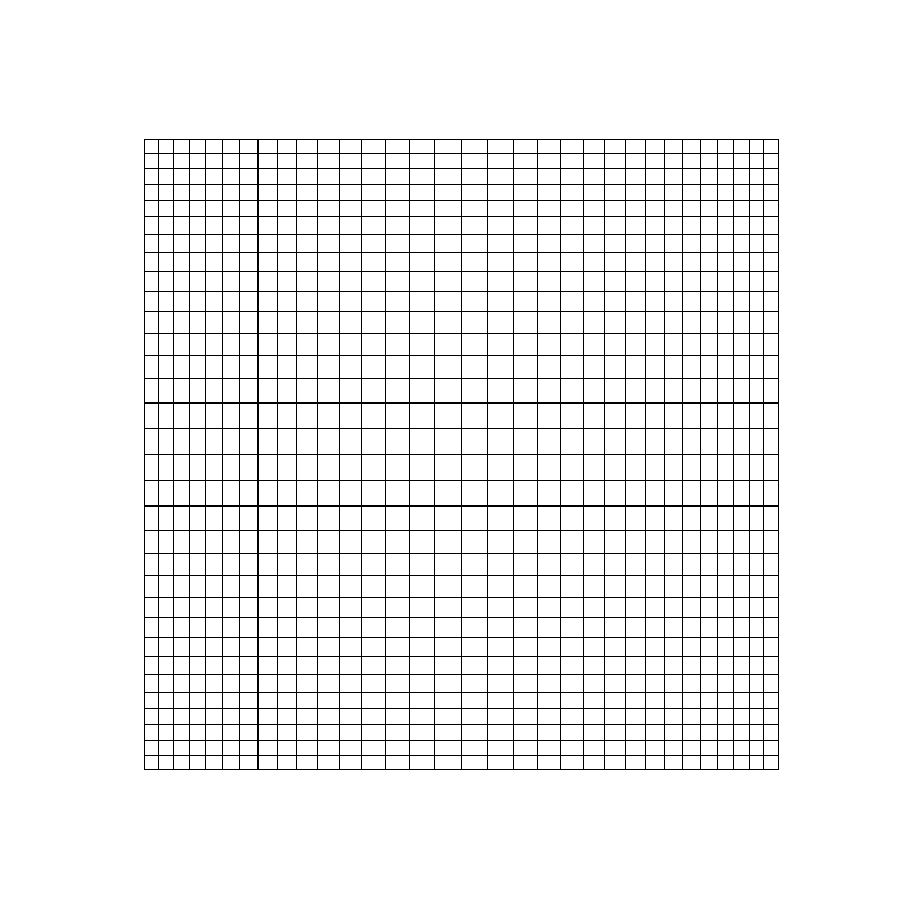} \\
  \includegraphics[width=0.25\textwidth]{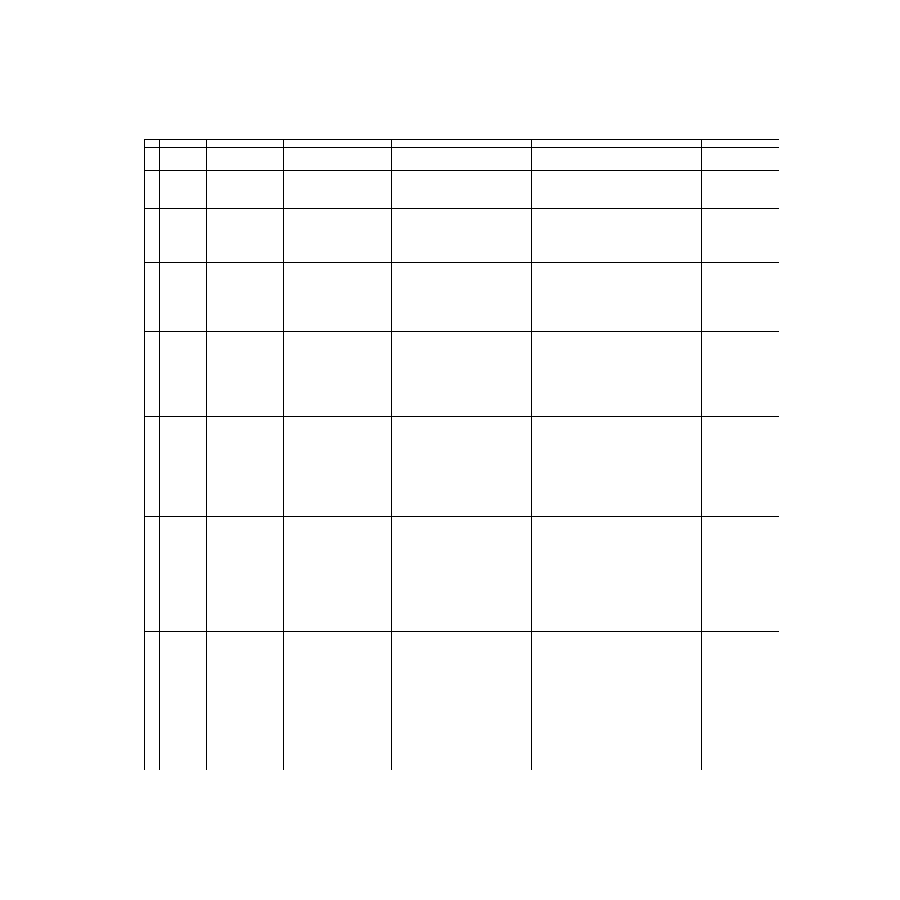} \hfil
  \includegraphics[width=0.25\textwidth]{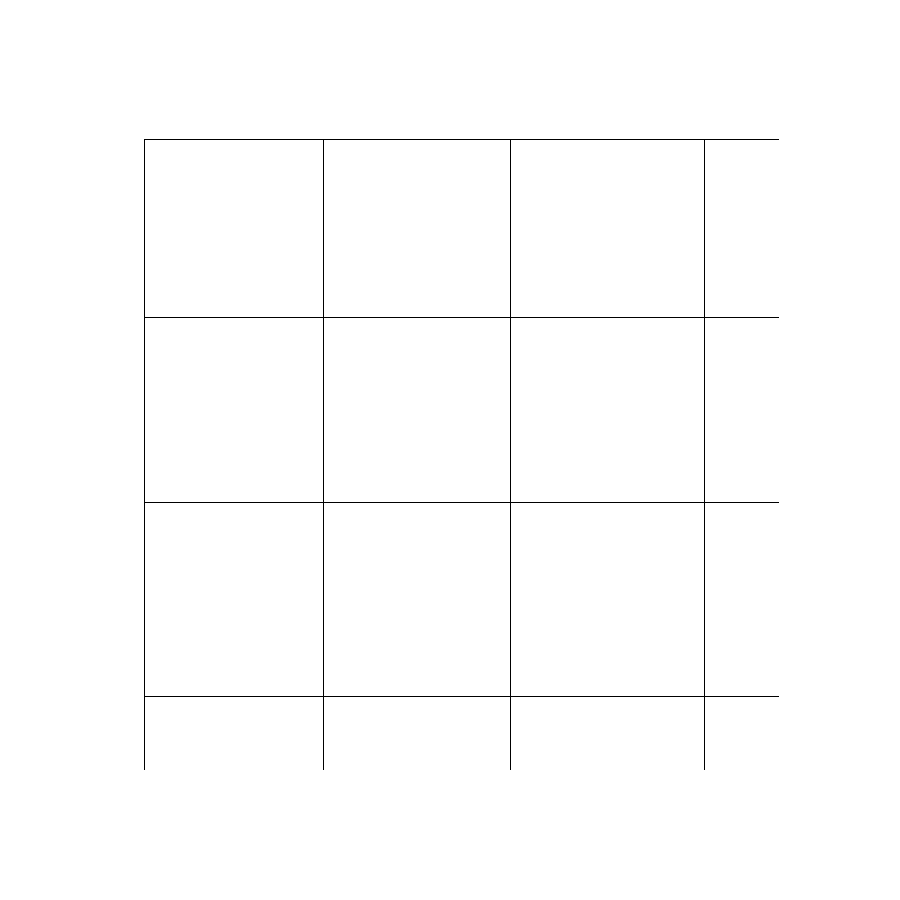}
  \caption{\label{fig:mesh_designs} Illustrations of the two types of graded
  mesh used. Both meshes consist of 1024 elements. Left: mesh
  $\mathcal{R}_{32}$, offering aggressive refinement near the boundary. Right:
  $\T{}_{32}$. Bottom row: zoom of upper left corners showing an area of
  $0.08 \times 0.08$ to illustrate the differences in grading.}
\end{figure}

Meshes constructed from central square elements with a constant reduction ratio
of 0.96 towards each boundary are denoted by $\T{}_N$, where $N$ is the
number of elements in each coordinate direction. We denote by $\mathcal{R}_N$
meshes constructed with linearly decreasing element size as follows (similar
to those used in \cite{pan2009simulation}). We define vertices for an $N \times
N$ quadrilateral mesh by their $x$ and $y$ coordinates, given respectively by
$x_i = 2 \left(\tfrac{i}{N}\right)^2$ for $i=0,1,...,\tfrac N 2$, with $x_i = 1
- x_{N-i}$ for $i = \tfrac N 2 +1, ..., N$, and $y_i = 1 - \left(1 - \tfrac i
N\right)^2$ for $i=0,...,N$. On these meshes, the finest vertical resolution
occurs at the lid, with an element height of $\tfrac{1}{N^2}$.

This means the meshes are highly anisotropic away from the
corners. Stabilised finite element methods often have stabilisation
parameters that depend badly on the anisotropy of the elements,
however the solver we used (\ref{eq:velocity_stabilised_galerkin}) is
parameter free.

\begin{table}
  \caption{Mesh statistics}
\begin{tabular}{ p{3cm}p{3cm}p{3cm}p{3cm} }
  \hline\hline
  Mesh & \# elements  & $h_{\text{min}}$ & $h_{\text{max}}$\\
  \hline
  $\T{}_{90}$   & 8,100  & 0.0039 & 0.024\\
  $\T{}_{120}$   & 14,400 & 0.0020 & 0.022 \\
  $\T{}_{150}$   & 22,500 & 0.0010 & 0.021 \\
  $\T{}_{180}$   & 32,400 & 0.00054 & 0.021 \\
  $\mathcal{R}_{256}$   & 65,536 & $ 1.5\times10^{-5}$ & $7.8 \times 10^{-3}$ \\
  \hline
\end{tabular}
\label{table:mesh_statistics}
\end{table}

\subsection{Numerical results}

Where possible, the results presented here are compared with the literature,
particularly \cite{fattal2005time,pan2009simulation,sousa2016lid} - see Table
\ref{table:comparison_metrics}. Descriptions of the meshes used are given in
\S\ref{sec:mesh_design}, and a summary of the key statistics of each is given in
Table \ref{table:mesh_statistics}. The data for comparison are the logarithm of the maximum value of $\sig_11$ along the midline, $x = 0.5$, the maximum value of $\sig_11$ over the domain, and the coordinates of the centre of the main vortex, $x_c, y_c$.

\subsubsection{\(\wi = 0.5\)}

The solution appears to have reached a steady state by $t = 8$. In Figure
\ref{fig:cross_sections_wi_05_fd}, components of the conformation tensor are
plotted along cross sections of the domain for $\wi = 0.5$ and for various
computational meshes described in \S\ref{sec:mesh_design} and Table
\ref{table:mesh_statistics}. The solutions obtained for $\sig_{11}$ and
$\sig_{22}$ agree qualitatively very well with Figure 5 of
\cite{pan2009simulation}, albeit with differences in magnitude along the lid. We
note that the meshes used there are finer, which could account for the
difference. Other metrics available for comparison from the numerical literature
(see Table 4 in \cite{sousa2016lid} and Table \ref{table:comparison_metrics})
are the centre of the main vortex in the velocity field and the logarithm of the
maximum value of $\sig_{11}$ attained along the line $x=0.5$. Computed on mesh
$\mathcal{R}_{256}$, the former agrees very closely (within 1\%) although not
much variation in prediction of this quantity is observed across the numerical
literature. The latter is approximately 2\% larger than the other predictions
given in \cite{sousa2016lid}. Resolution of the layer at the top boundary
appears to be of critical importance for the accuracy of the approximation.
Under-resolution results in altered behaviour of $\sig_{12}$ at the upper
boundary. In addition, coarse meshes appear to over-estimate $\sig_{11}$. These
errors do not appear to pollute the solution away from the boundary.

\begin{figure}
  \includegraphics[width=0.33\textwidth]{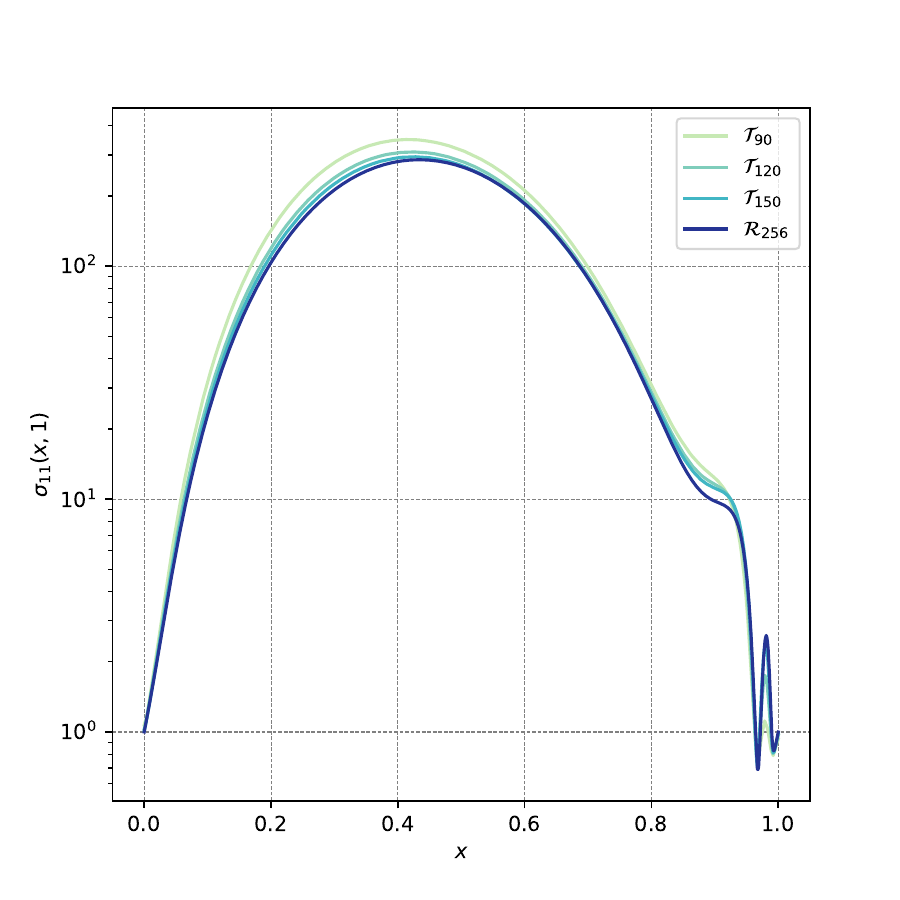}
  \includegraphics[width=0.33\textwidth]{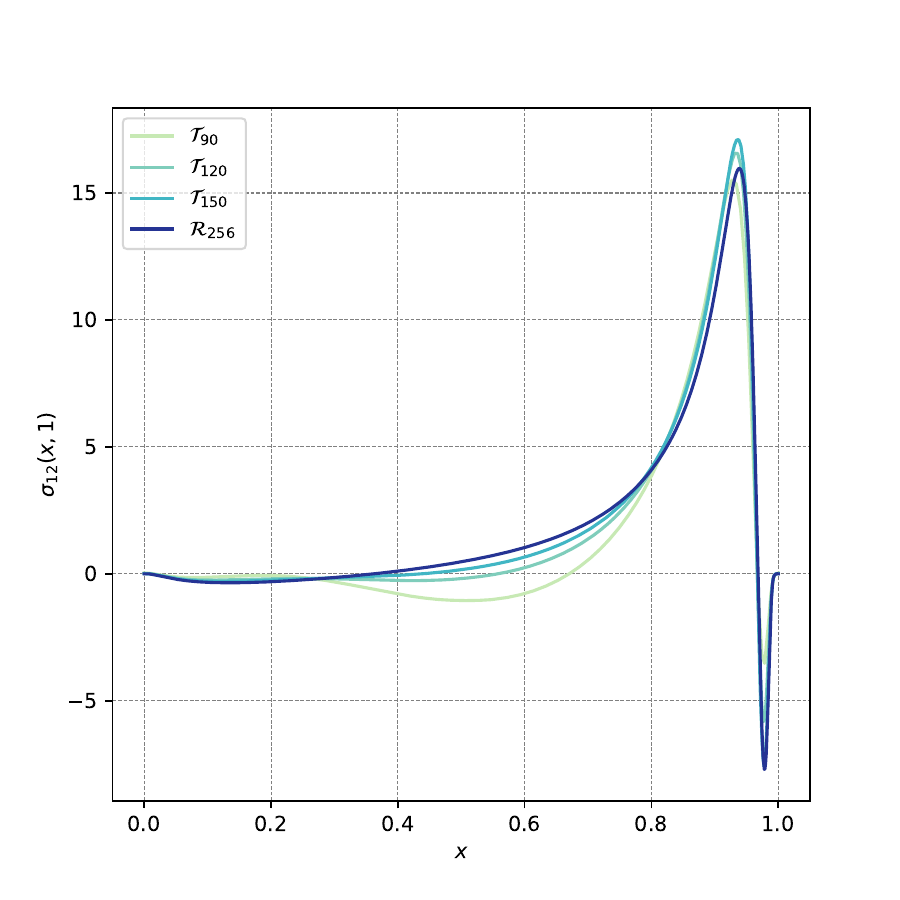} 
  \includegraphics[width=0.33\textwidth]{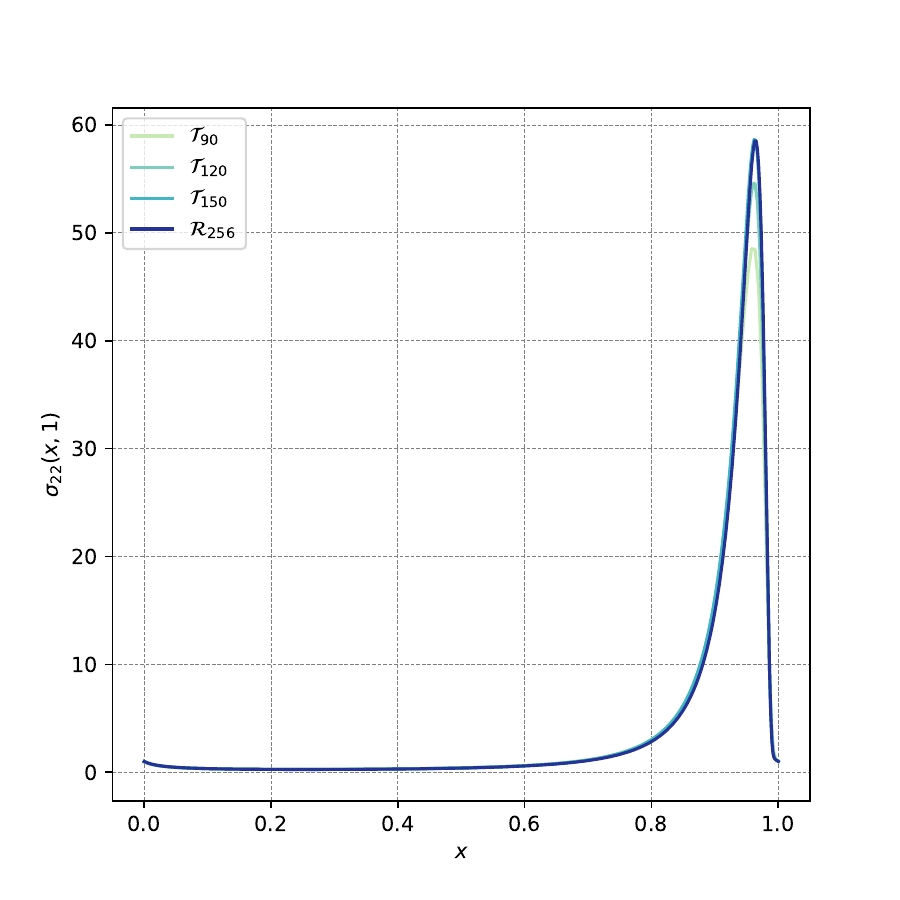} \\
  \includegraphics[width=0.33\textwidth]{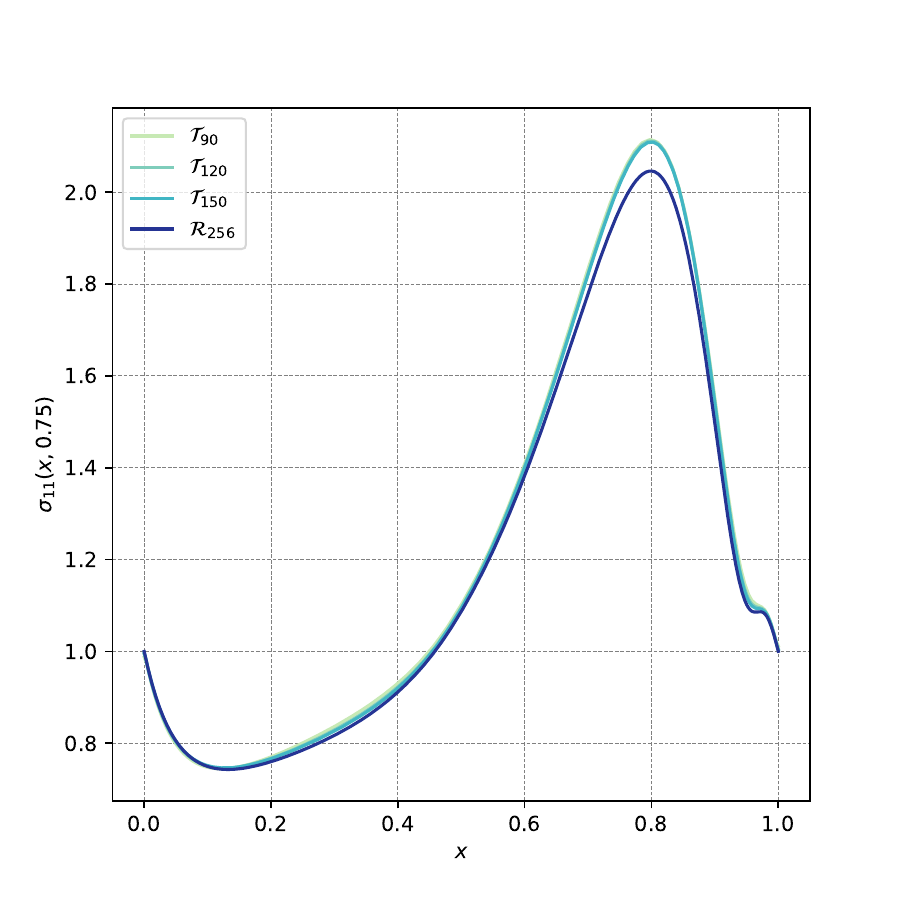}
  \includegraphics[width=0.33\textwidth]{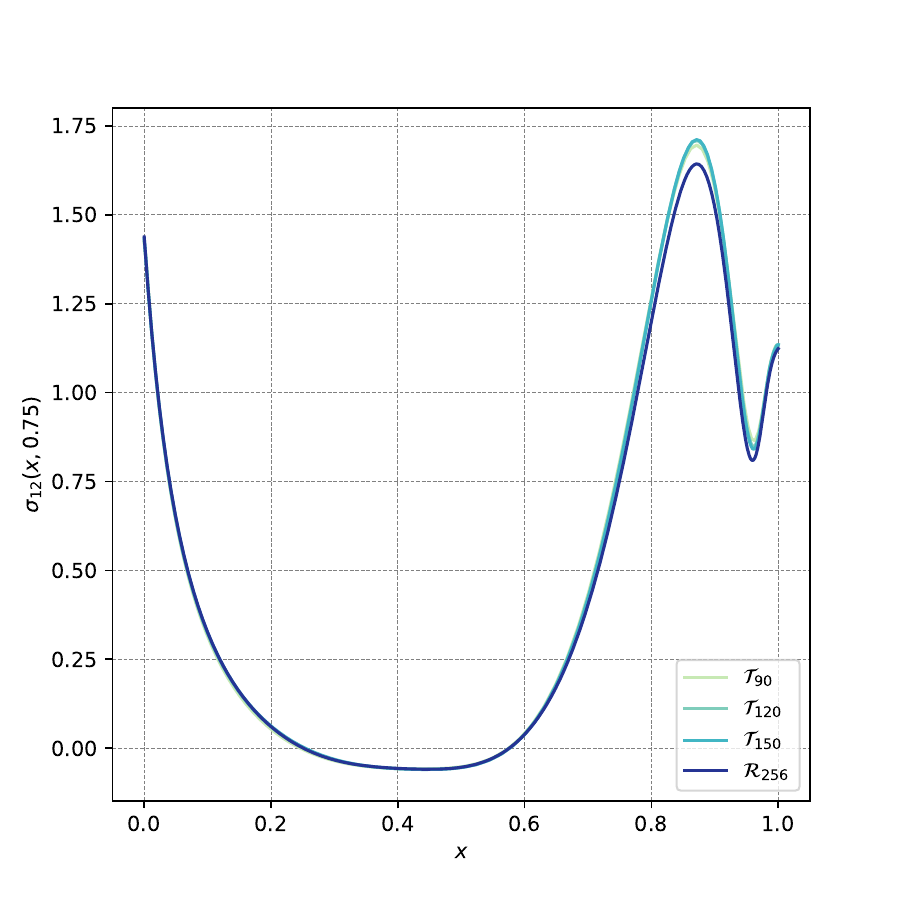} 
  \includegraphics[width=0.33\textwidth]{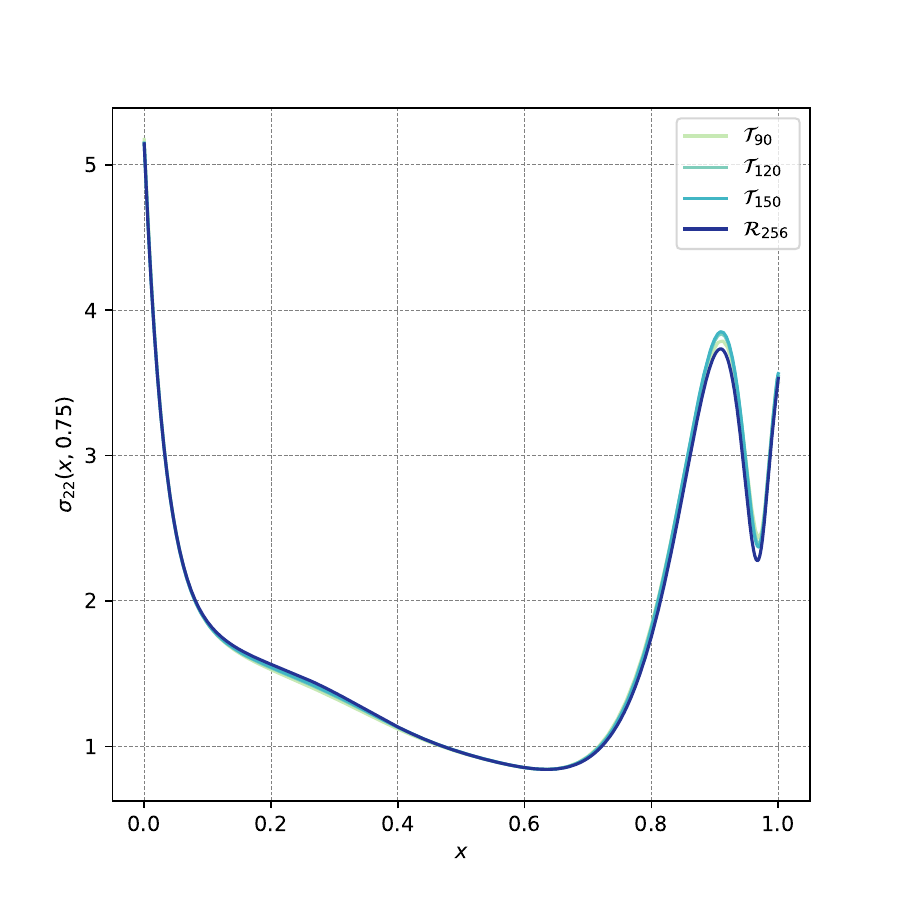} \\
  \includegraphics[width=0.33\textwidth]{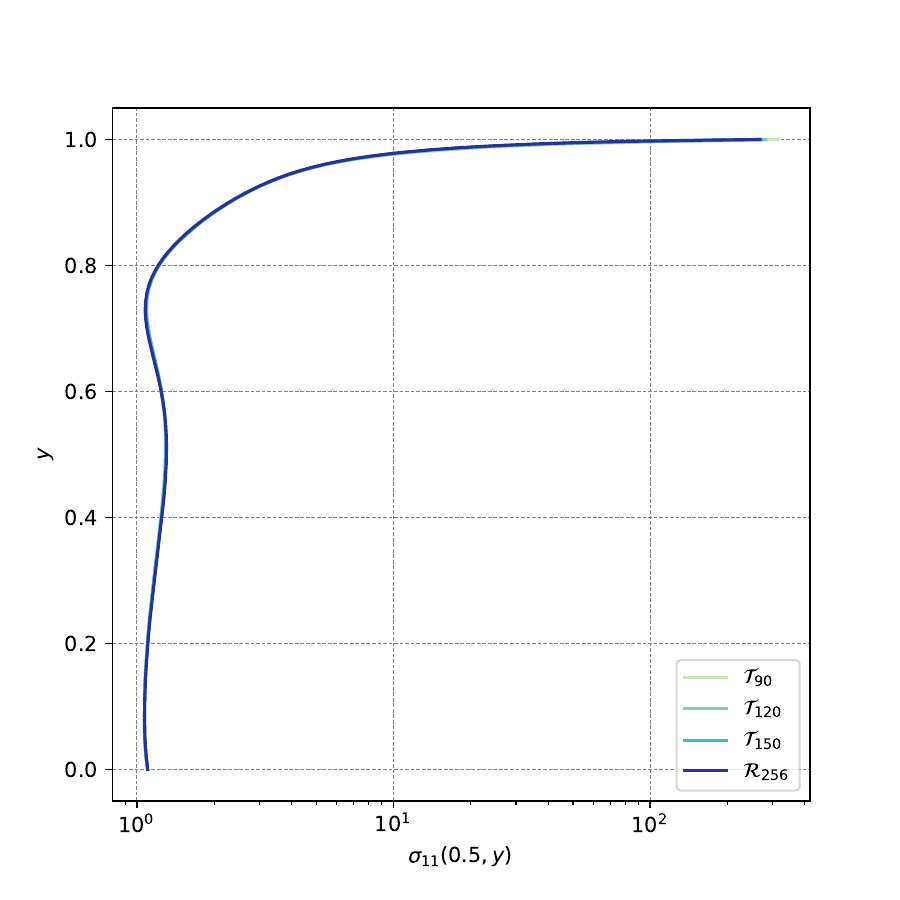}
  \includegraphics[width=0.33\textwidth]{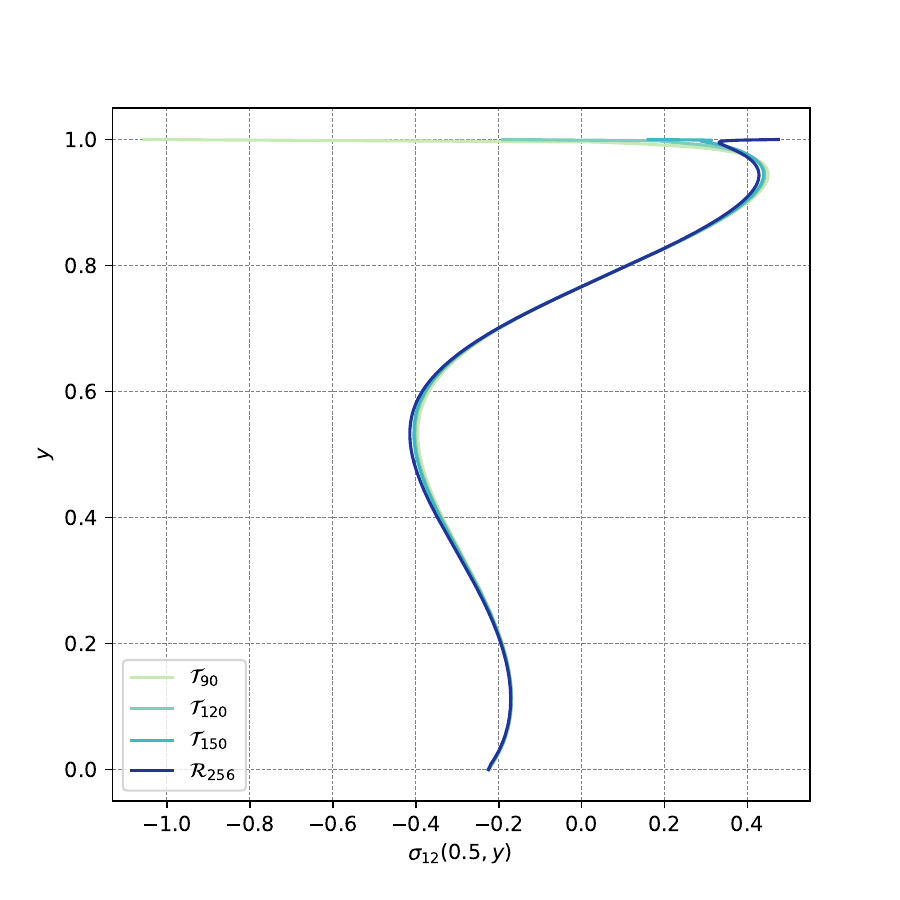} 
  \includegraphics[width=0.33\textwidth]{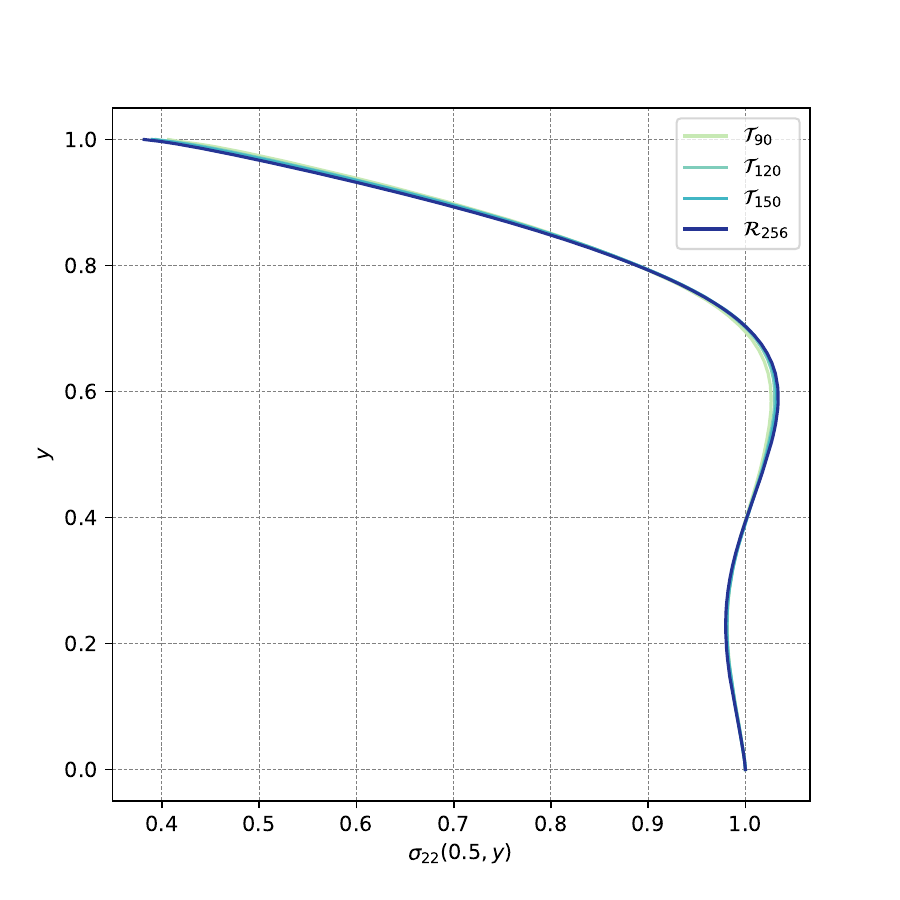}
  \caption{\label{fig:cross_sections_wi_05_fd} Components of the conformation
  tensor plotted along cross sections of the domain, obtained on meshes
  $\T{}_{90}$, $\T{}_{120}$, $\T{}_{150}$ and
  $\mathcal{R}_{256}$ with $\wi = 0.5$ and $t=10$. Top row: plots over the top
  boundary given by $y = 1$. Middle row: plots over the line given by $y=0.75$.
  Bottom row: plots over the midline given by $x=0.5$. }
\end{figure}

\begin{table}
  \caption{Comparison of results with published data for the (regularised)
  lid-driven cavity, $\beta = 0.5$. Value marked $^*$ was estimated from
  \cite{pan2009simulation} using their predicted relationship between this
  quantity and $\wi$ (see \cite[\S 4.3]{pan2009simulation}), while values marked
  $^\dagger$ were estimated using a plot digitisation tool \cite{plotdigitizer}.
  All other values were reported directly. All values from the current work were
  taken after a steady state had been reached (approx $ t = 10$ for $\wi = 0.5$
  and $t=30$ for $\wi = 1$).}
\begin{tabular}{ p{4cm}p{3cm}p{3cm}p{3cm} }
  \hline\hline
  Reference &$\underset{x = 0.5}{\max}\ln(\sig_{11})$  & $\max \sig_{11}$ &
  $x_c, y_{c}$\\[1ex]
  \hline
  $\wi = 0.5$ & & \\
  \hline
  Current work $\T{}_{90}$   &5.76 & 351.21 & 0.466, 0.799\\
  Current work $\T{}_{120}$  & 5.65 & 309.65 & 0.466, 0.799 \\
  Current work $\T{}_{150}$  & 5.60 & 295.04 & 0.467, 0.799  \\
  Current work $\mathcal{R}_{256}$  & 5.59 & 290.59 & 0.467, 0.799  \\
  Pan et al. \cite{pan2009simulation}    & $5.59^{\dagger}$ & $289^*$ &
  0.469, 0.798  \\
  Sousa et al. \cite{sousa2016lid} M4  & 5.51 & - & 0.466, 0.800\\
  \hline
  $\wi = 1$ & & \\
  \hline
  Current work $\T{}_{180}$  & 10.75 & 48,038.9  & 0.428, 0.819 \\
  Current work $\mathcal{R}_{256}$  & 10.22 & 28,069.2  & 0.431, 0.819
  \\
  Pan et al. \cite{pan2009simulation}  & $9.35^{\dagger}$ & 11,529.43 & 0.439,
  0.816  \\
  Sousa et al. \cite{sousa2016lid} M4  & 7.80 & - & 0.434, 0.816\\
  \hline
\end{tabular}
\label{table:comparison_metrics}
\end{table}

\subsubsection{\(\wi = 1\)}

The solution appears to have reached a steady state by $t = 27.5$. The
conformation tensor field is plotted component-wise in Figure
\ref{fig:conformation_tensor_fields}, with benchmark metrics reported in Table
\ref{table:comparison_metrics}. The logarithm of the conformation tensor was
computed to allow comparison with other published works, with selected cross
sections presented in Figure \ref{fig:comparison_with_others}. A sharp boundary
layer at the lid is observed in $\sig_{11}$, with all components exhibiting
large gradients near the upper corners of the domain. Our computational results
agree qualitatively with those presented in \cite[Figure 4]{fattal2005time},
although in our simulations a steady state was reached much later than $t=8$, as
reported there. In addition, the maximum value $\sig_{11}$ along $x=0.5$ was
significantly overestimated in our work relative to other published data, but
there is significant variation in this figure, with \cite{pan2009simulation}
reporting a value two logarithmic orders larger than \cite{fattal2005time}.
However, it can be seen from Figure \ref{fig:comparison_with_others} that large
discrepancies occur only in a very small neighbourhood of the upper boundary.
Qualitative differences are observed only in the upper right hand corner.

\begin{figure}
  \includegraphics[width=0.33\textwidth]{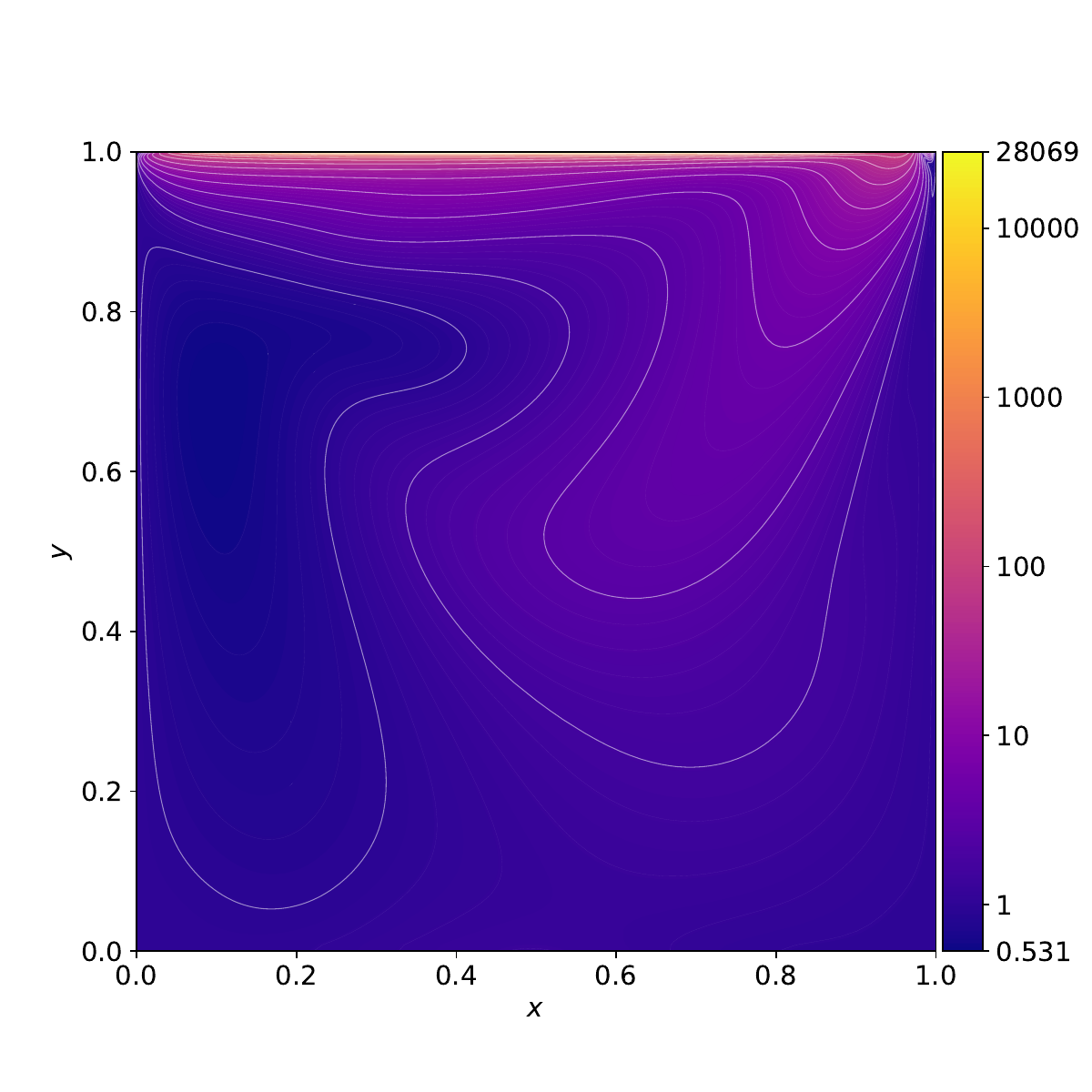}
  \includegraphics[width=0.33\textwidth]{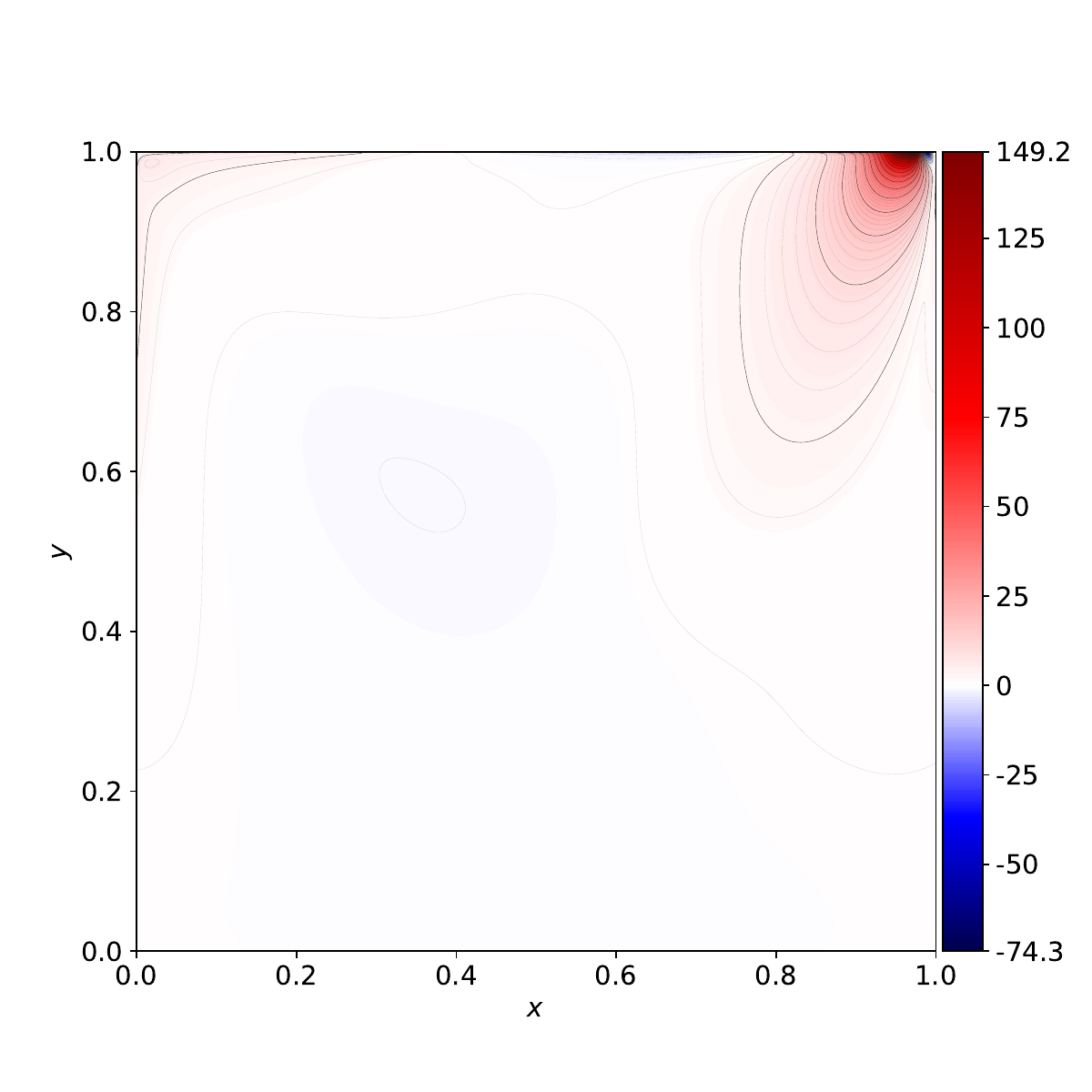} 
  \includegraphics[width=0.33\textwidth]{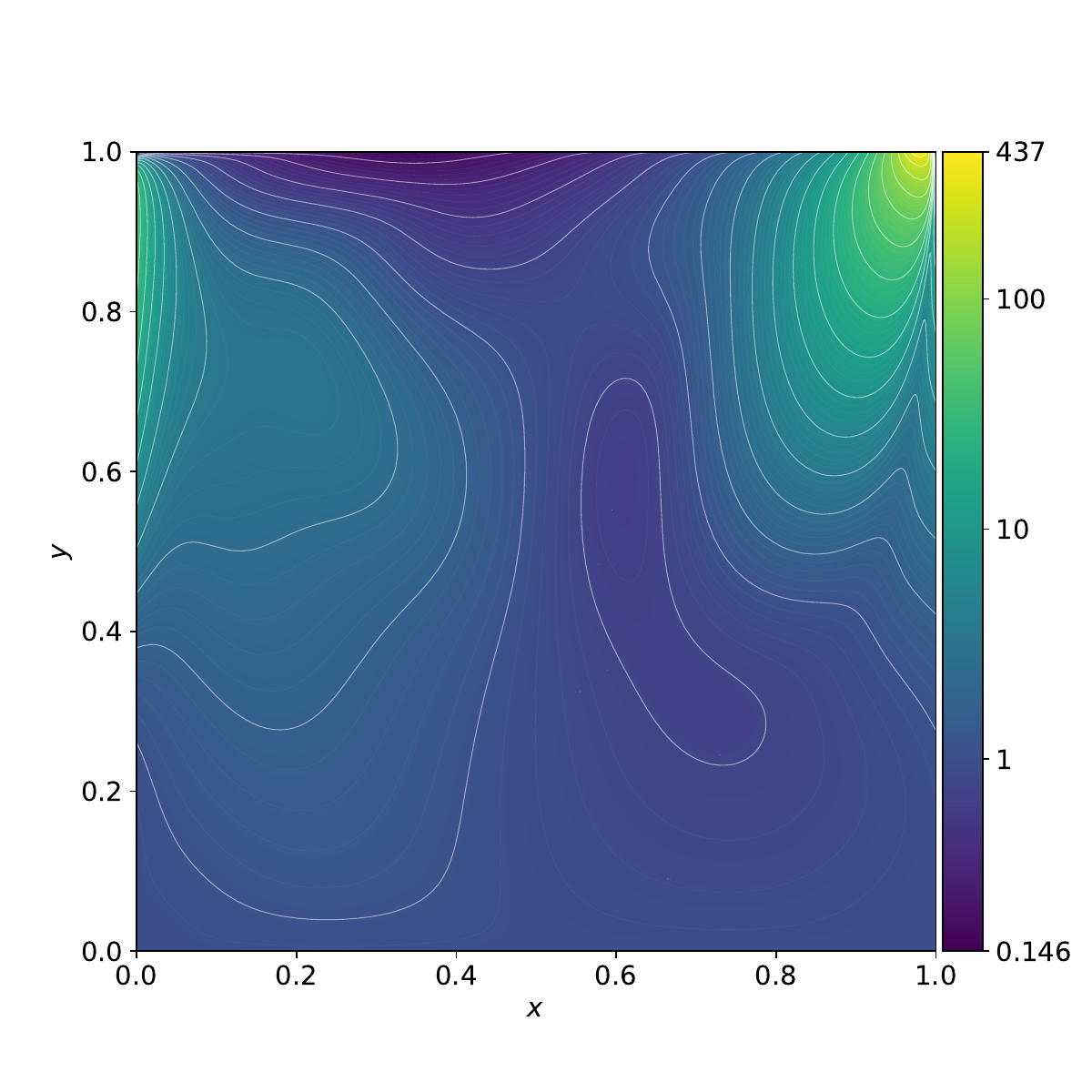}
  \caption{\label{fig:conformation_tensor_fields} $\wi = 1$. Components of the
  conformation tensor at $t=30$: $\sig_{11}$, $\sig_{12}$, $\sig_{22}$ (note $\sig_{11}$ and $\sig_{22}$ are
  logarithmically coloured and contoured while $\sig_{12}$ is linearly
  coloured and contoured). Different colormaps are used to emphasise the
  difference in magnitude between the components. Minimum and maximum values are
  shown on the colour bar.}
  \end{figure}

\begin{figure}
  \centering
  \includegraphics[width=0.4\textwidth]{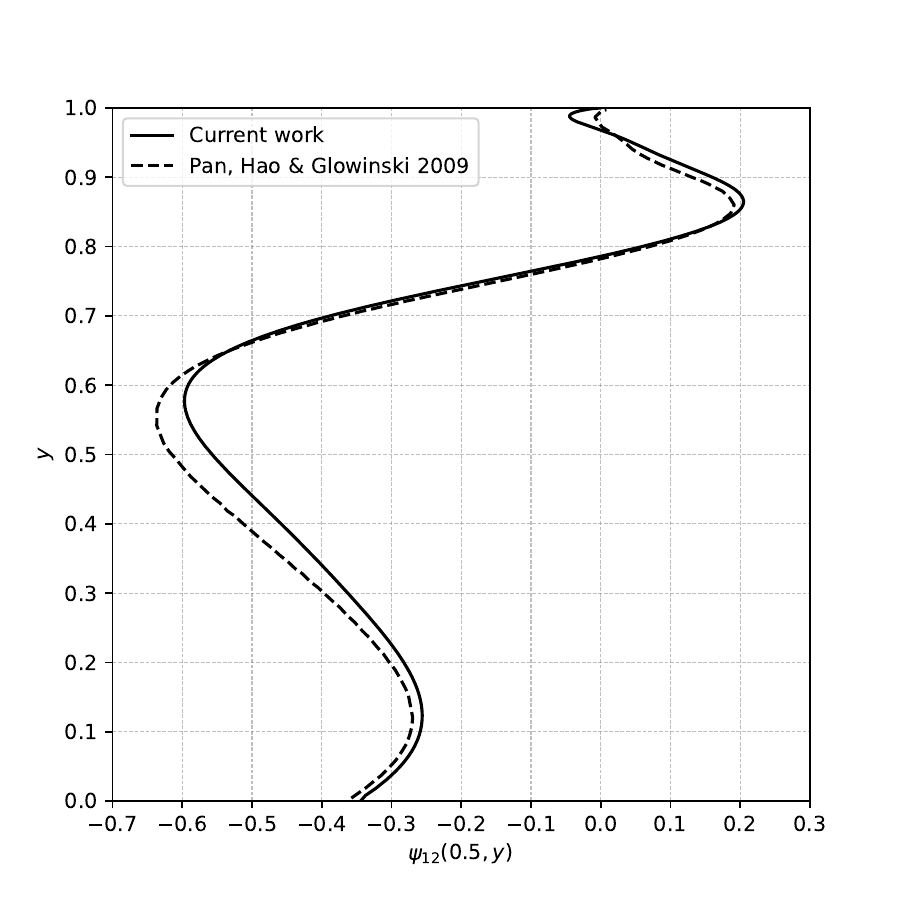}
  \includegraphics[width=0.4\textwidth]{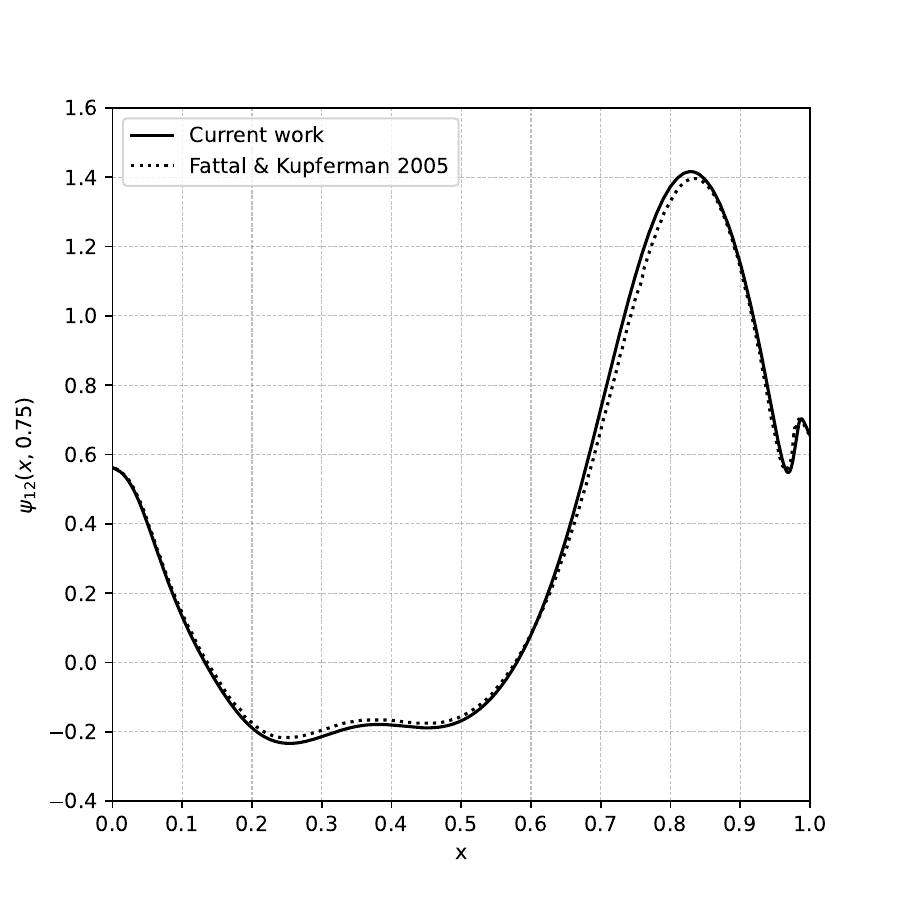}\\
  \includegraphics[width=0.4\textwidth]{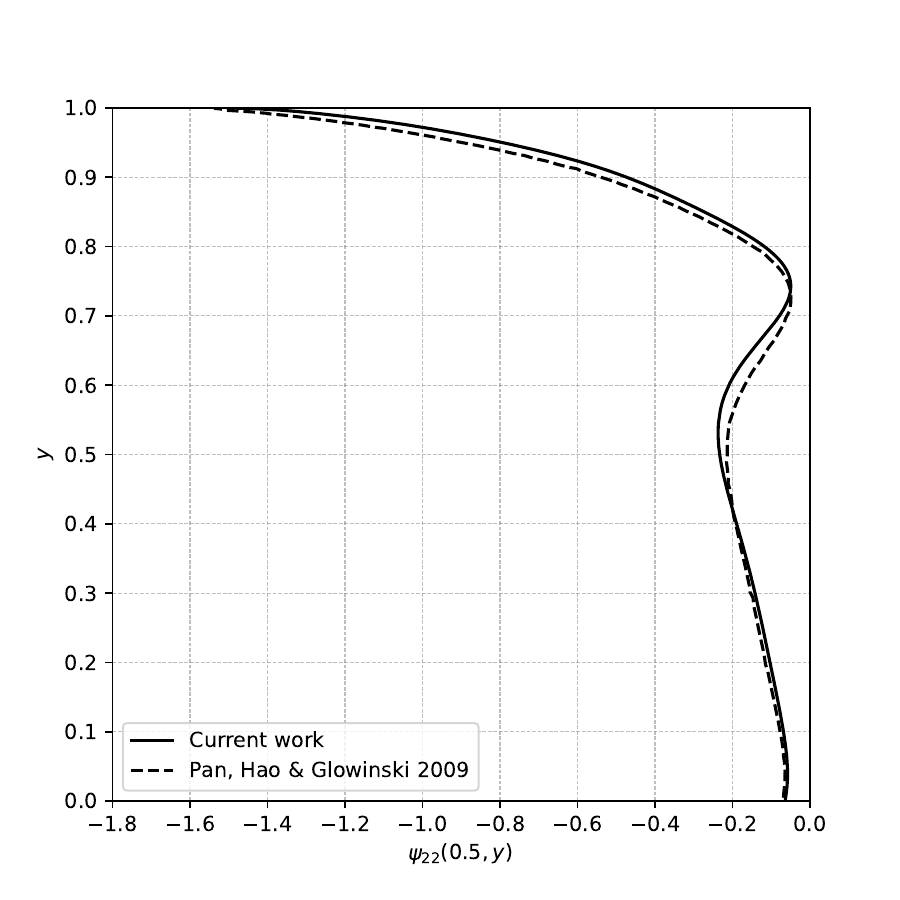} 
  \includegraphics[width=0.4\textwidth]{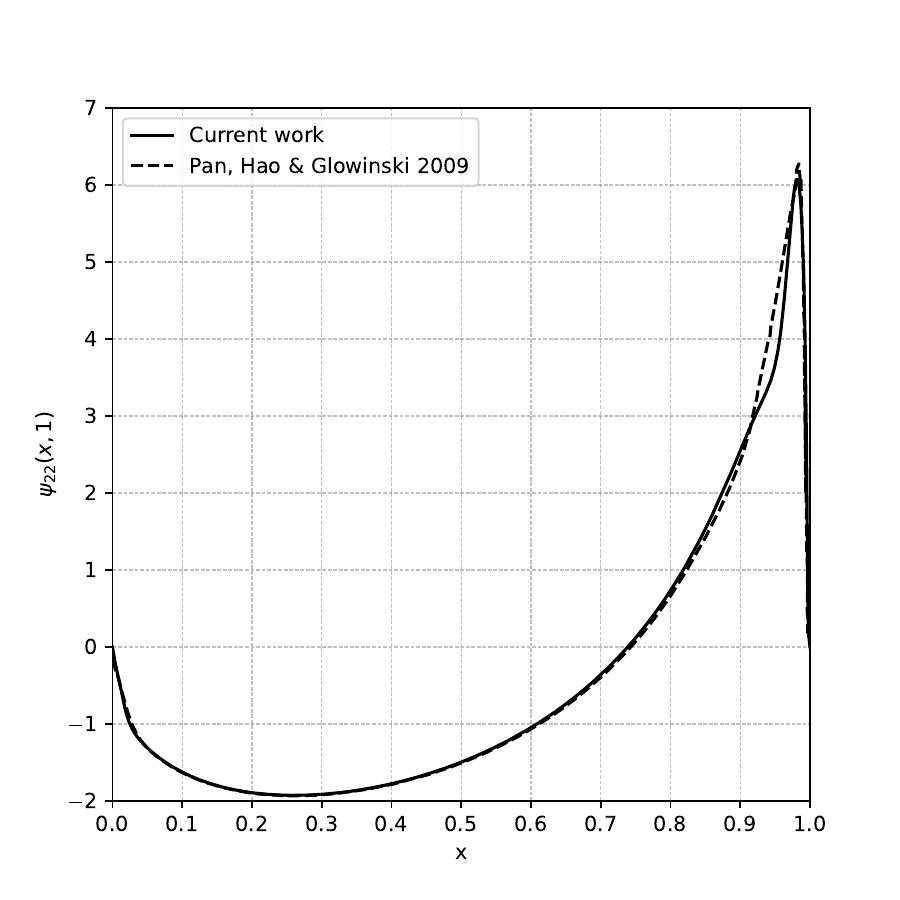}\\
  \includegraphics[width=0.4\textwidth]{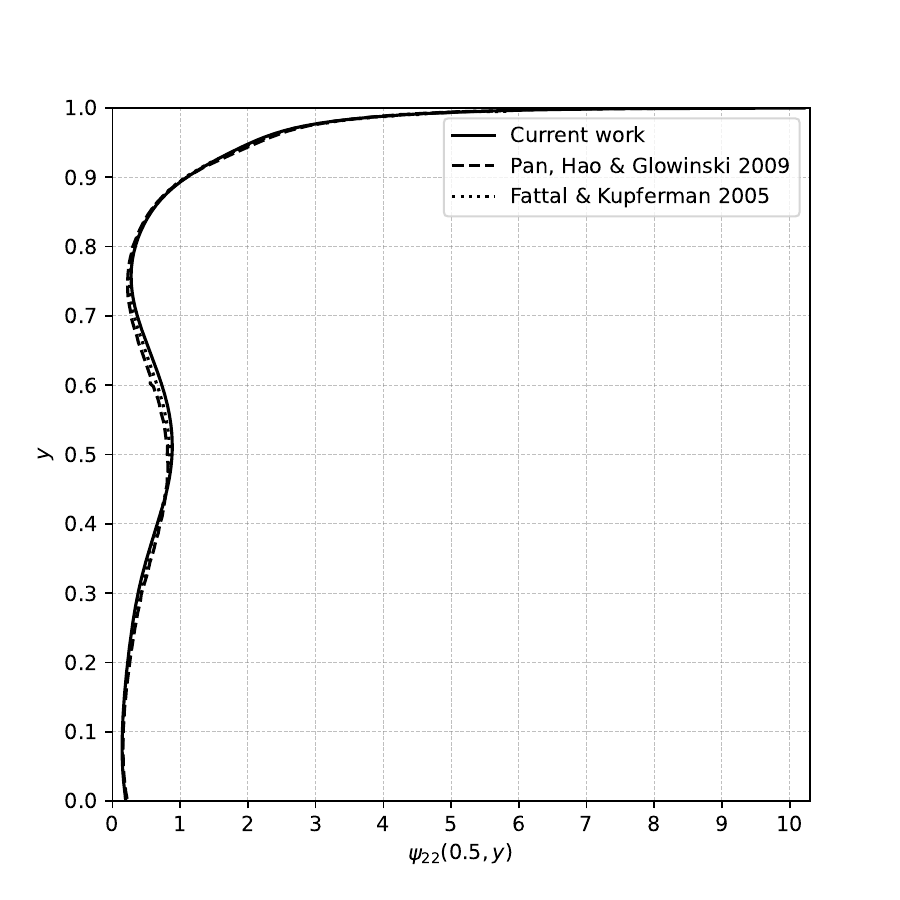}
  \includegraphics[width=0.4\textwidth]{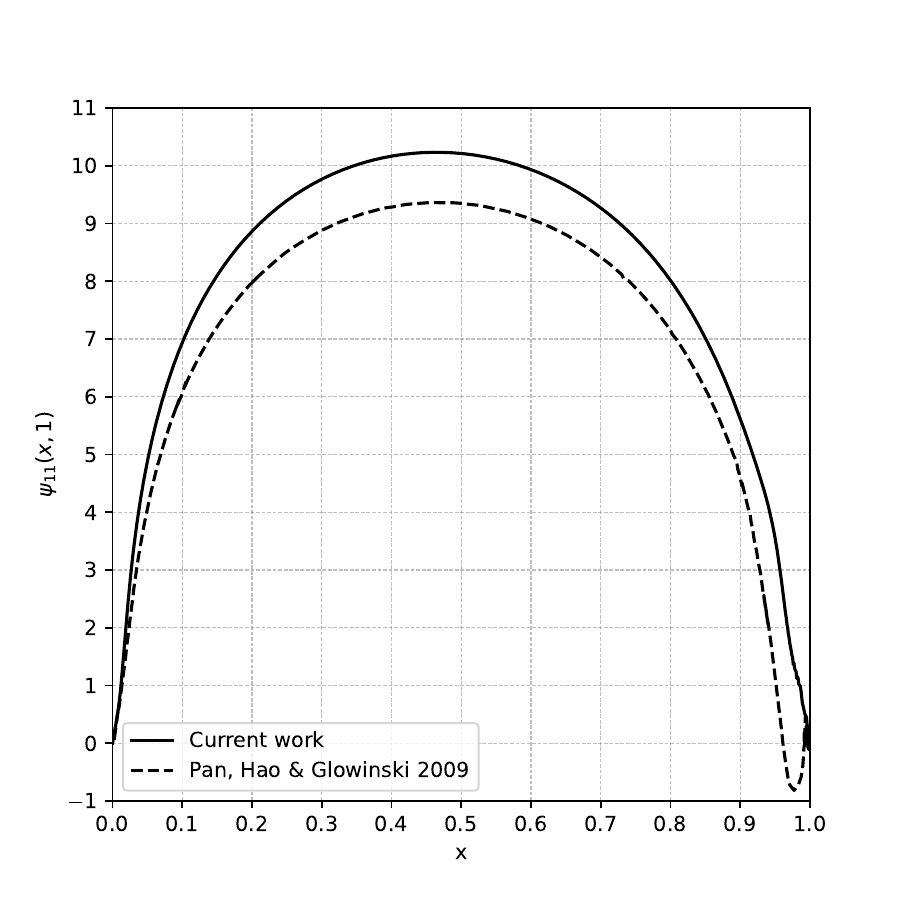}
  \caption{\label{fig:comparison_with_others} Components of the logarithm of the
  conformation tensor plotted along cross sections of the domain obtained on the
  mesh $\mathcal{R}_{256}$ with $\wi = 1$ (solid line). Dotted line is numerical
  data from Fattal \& Kupferman (2005) \cite{fattal2005time}, while the dashed
  line is data from Pan, Hao \& Glowinski (2009) \cite{pan2009simulation}.}
\end{figure}

In Figure \ref{fig:cross_sections_comparison}, components of the conformation
tensor are plotted along cross sections for two Weissenberg numbers, 0.5 and 1,
to illustrate the difference in solutions. Larger magnitudes and exaggerated
features are observed in all components for $\wi = 1$ when compared with $\wi =
0.5$.

\begin{figure}
  \includegraphics[width=0.33\textwidth]{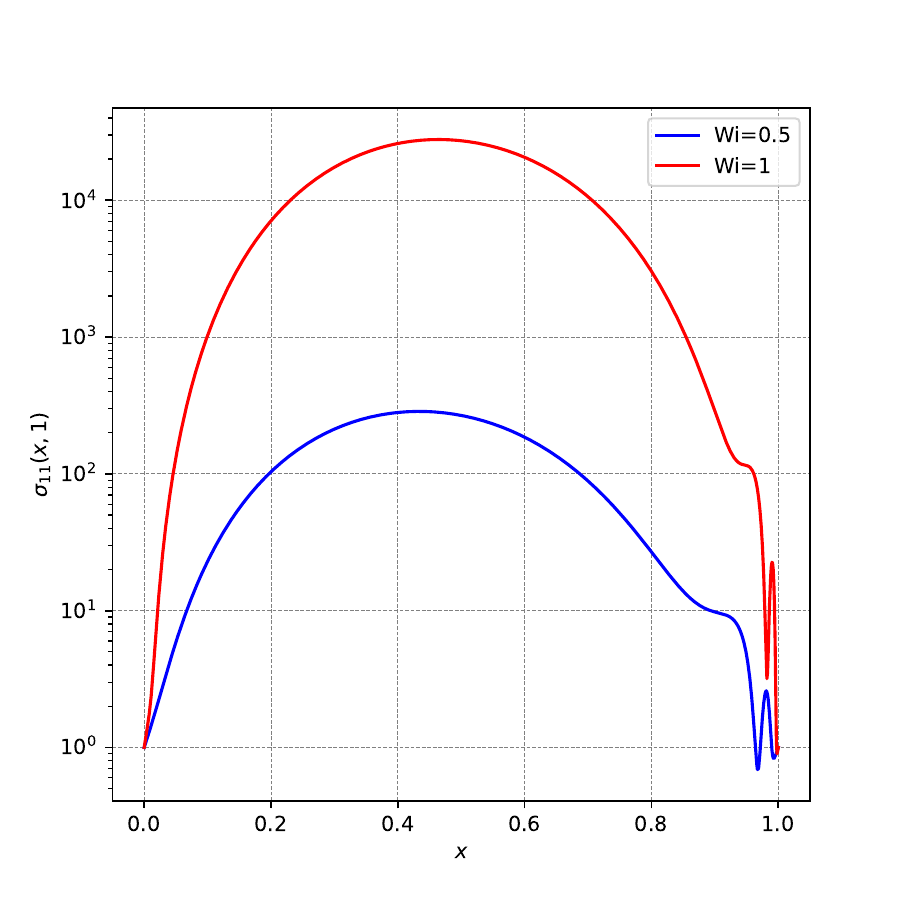}
  \includegraphics[width=0.33\textwidth]{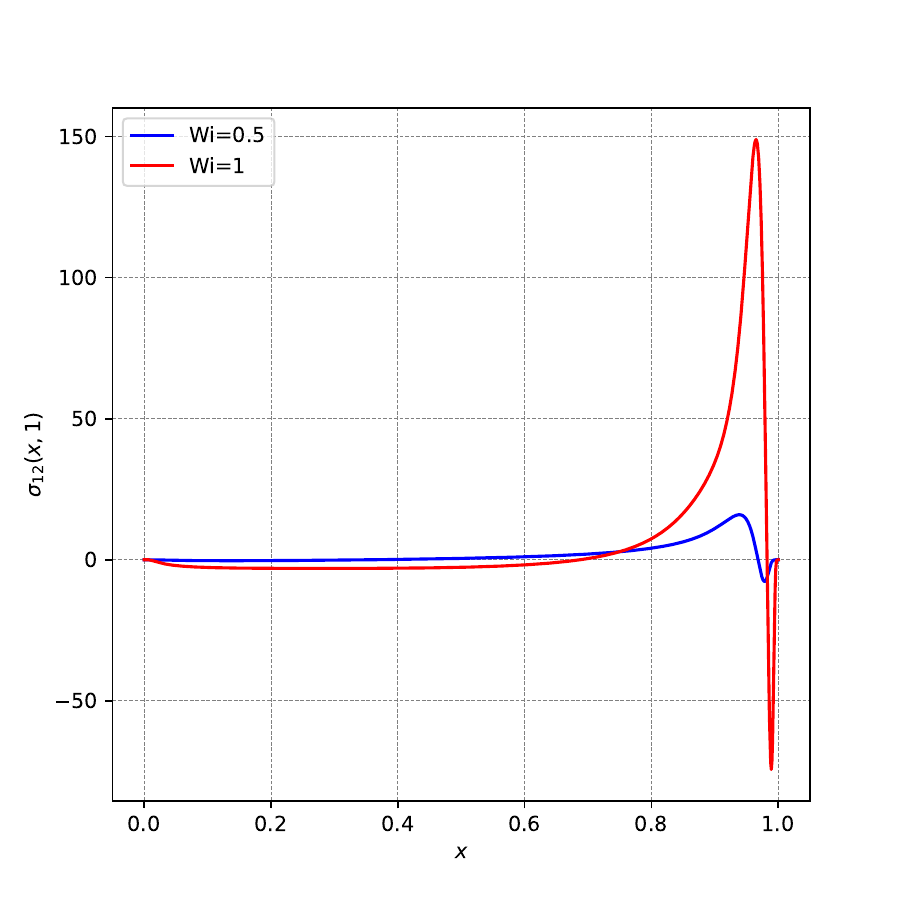} 
  \includegraphics[width=0.33\textwidth]{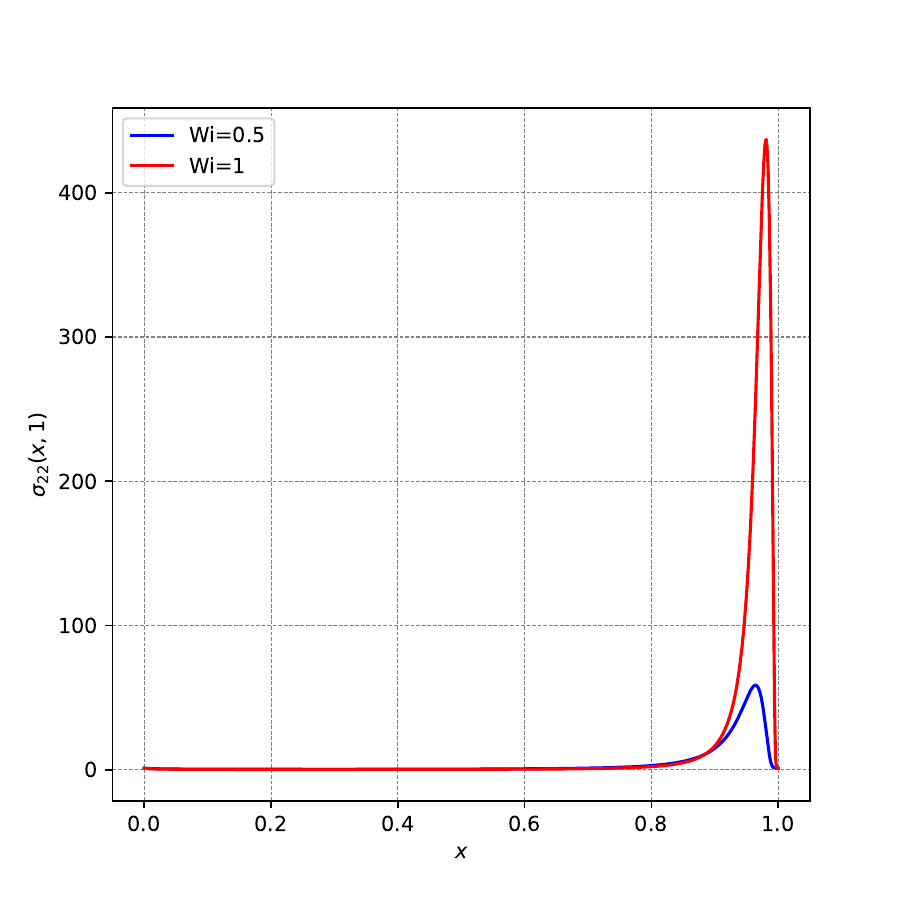} \\
  \includegraphics[width=0.33\textwidth]{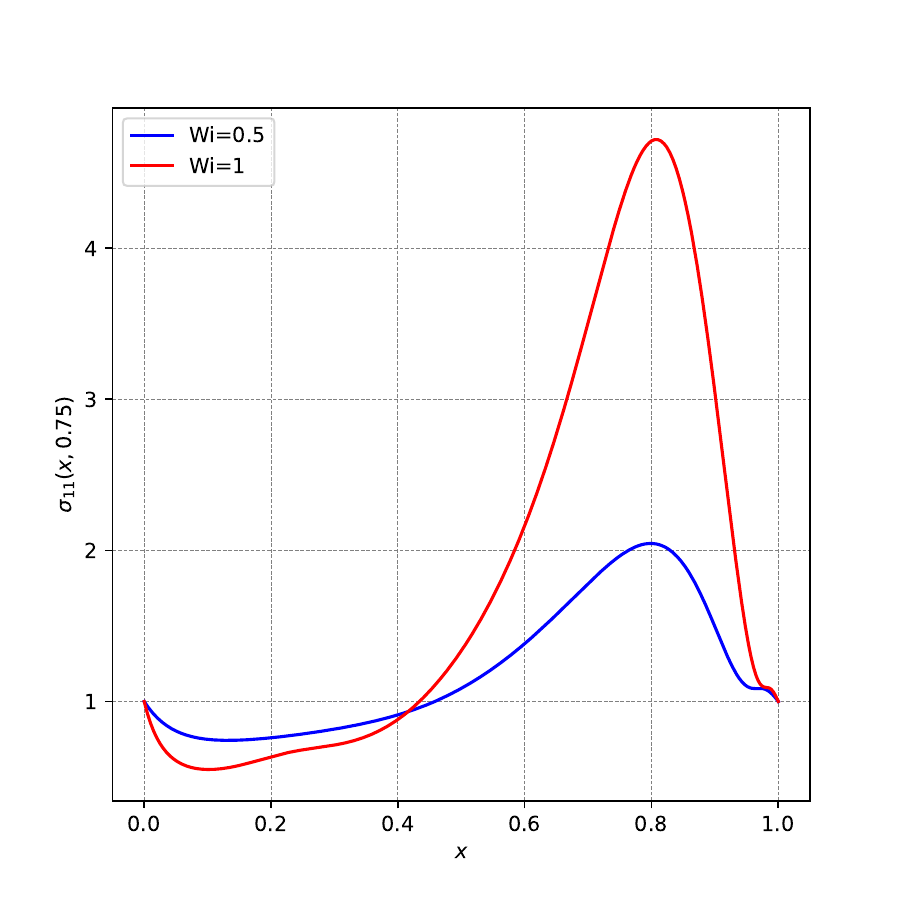}
  \includegraphics[width=0.33\textwidth]{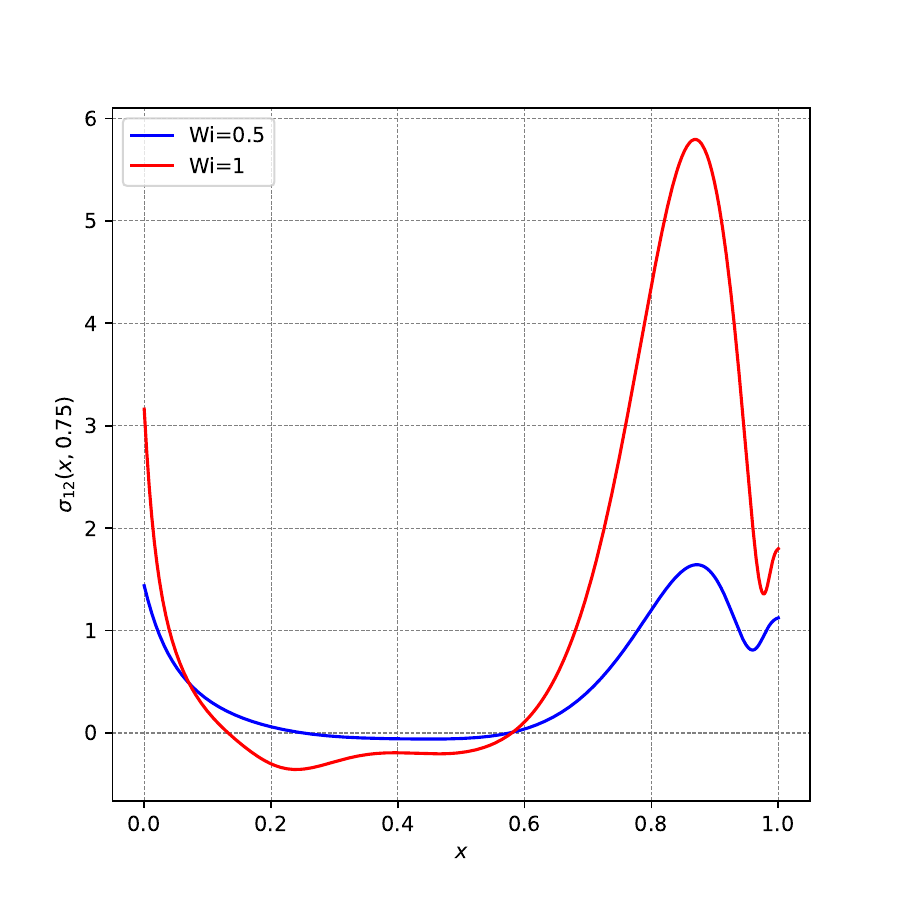} 
  \includegraphics[width=0.33\textwidth]{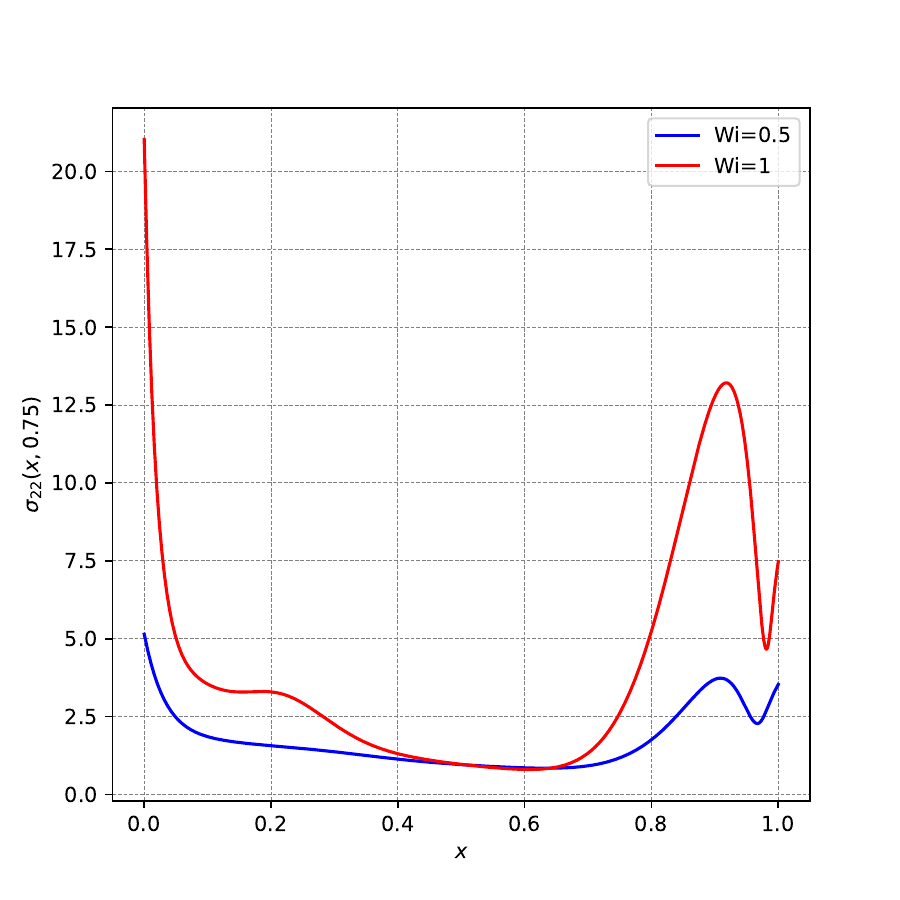} \\
  \includegraphics[width=0.33\textwidth]{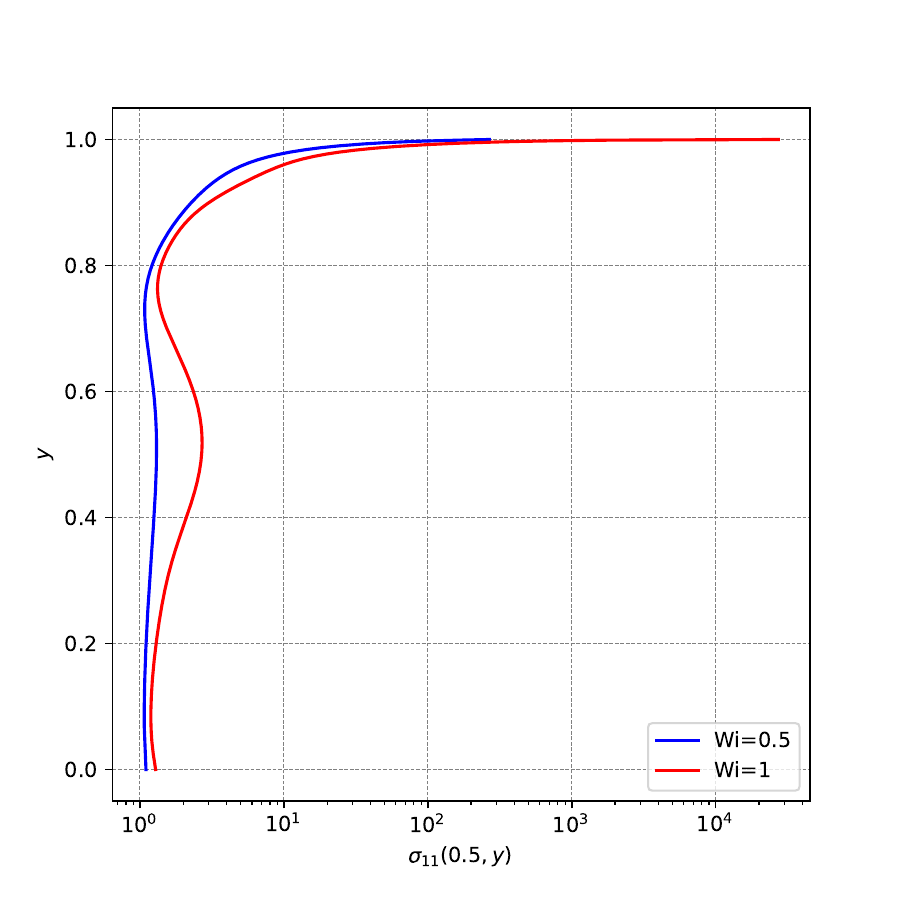}
  \includegraphics[width=0.33\textwidth]{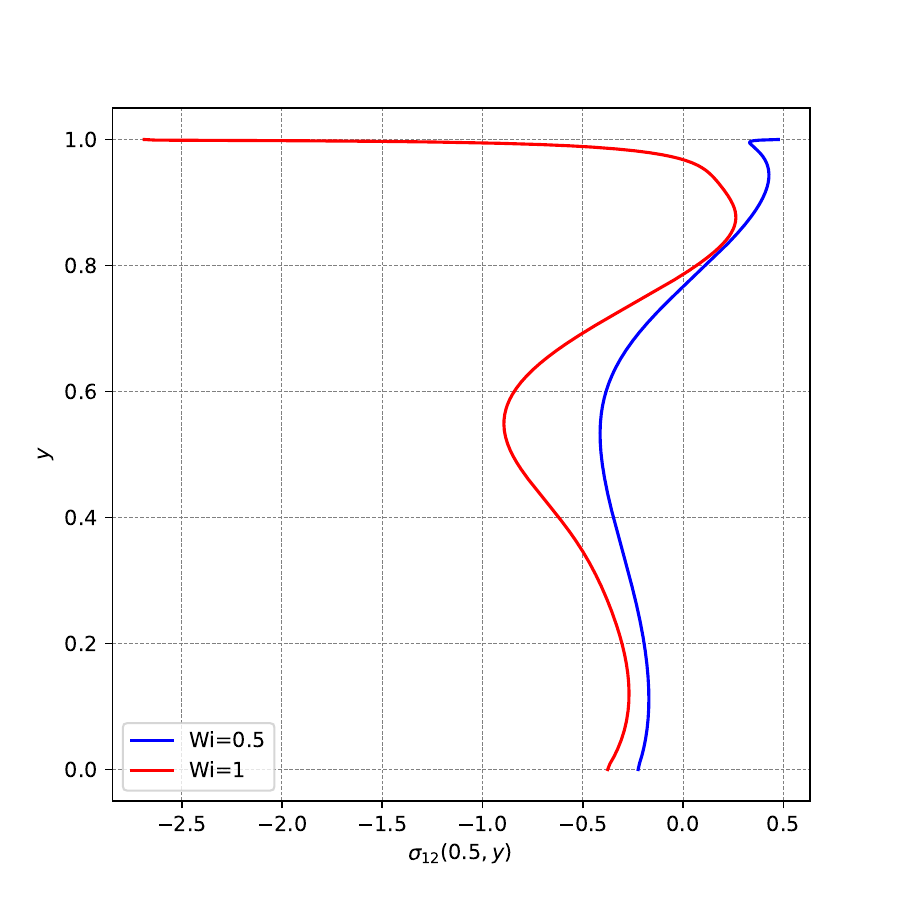} 
  \includegraphics[width=0.33\textwidth]{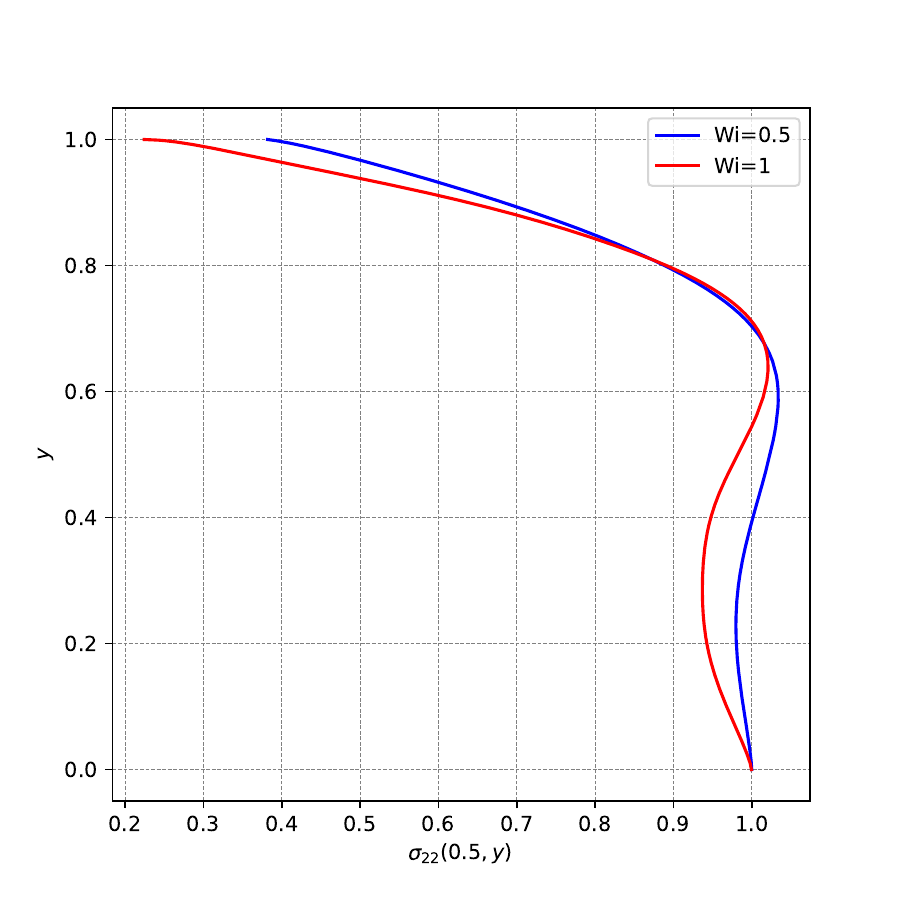}
  \caption{\label{fig:cross_sections_comparison} Components of the conformation
  tensor plotted along cross sections of the domain obtained on the mesh
  $\mathcal{R}_{256}$ with $\wi = 0.5, 1$. Top row: plots over the top boundary
  given by $y = 1$. Middle row: plots over the line given by $y=0.75$. Bottom
  row: plots over the midline given by $x=0.5$. }
\end{figure}

\subsection{Positive definiteness of the discrete conformation tensor}

A major aim of the numerical scheme presented is to maintain a positive definite
comformation tensor, both to respect the physics of the problem and avoid
numerical instabilities that can arise when this property is lost. We recall
from Proposition \ref{prop:pos_def_disc} that this is guaranteed if the time
step is chosen to be sufficiently small. To ensure that an appropriate choice
was made, the eigenvalues of the discrete conformation tensor were computed, and
indeed at all times, the smallest eigenvalue of the conformation tensor was
strictly positive, and therefore the discrete conformation tensor remained
positive definite throughout the computation. The eigenvalues of the
conformation tensor are visualised in Figure \ref{fig:eigenvalues}.

\begin{figure}
  \centering
  \includegraphics[width=0.33\textwidth]{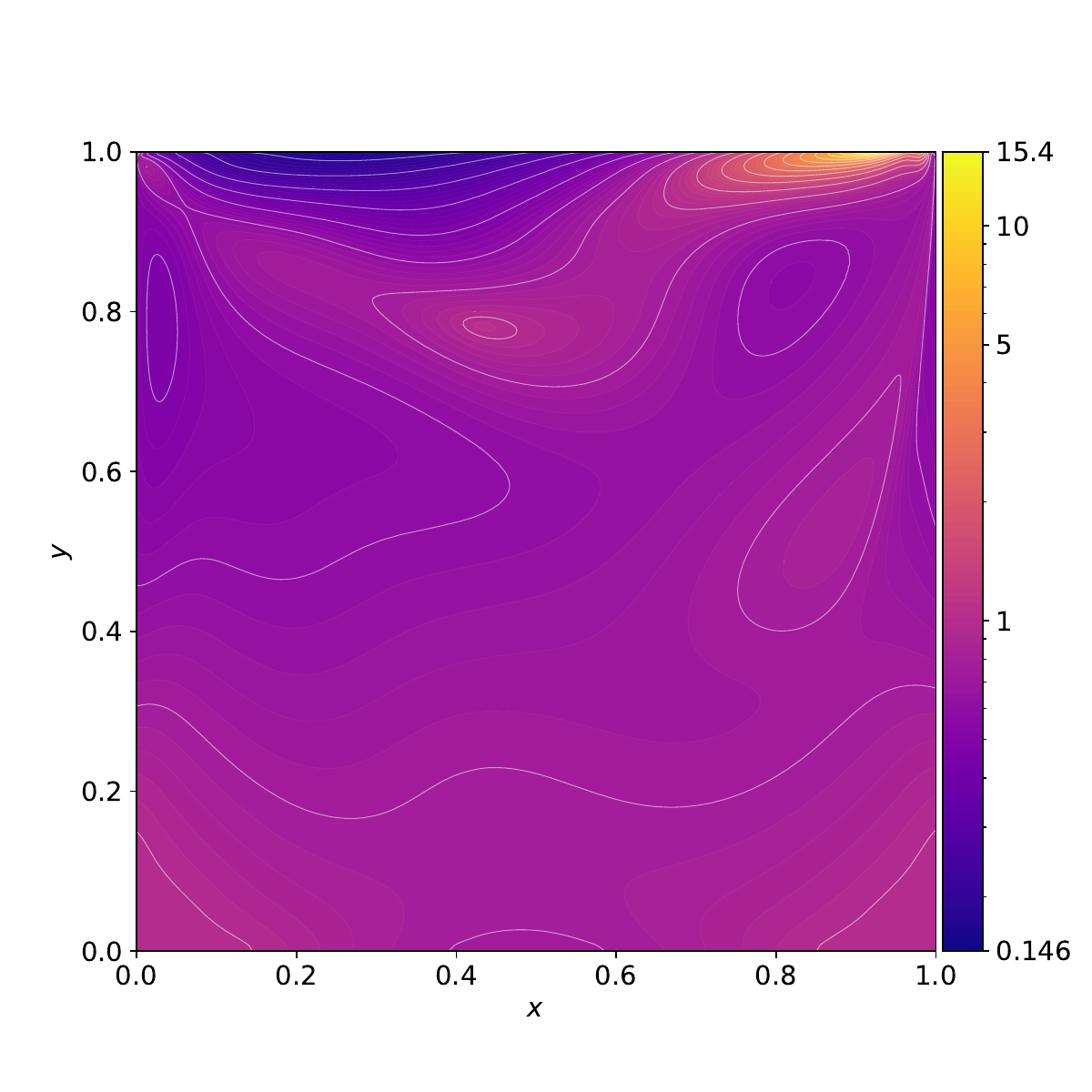}
  \includegraphics[width=0.33\textwidth]{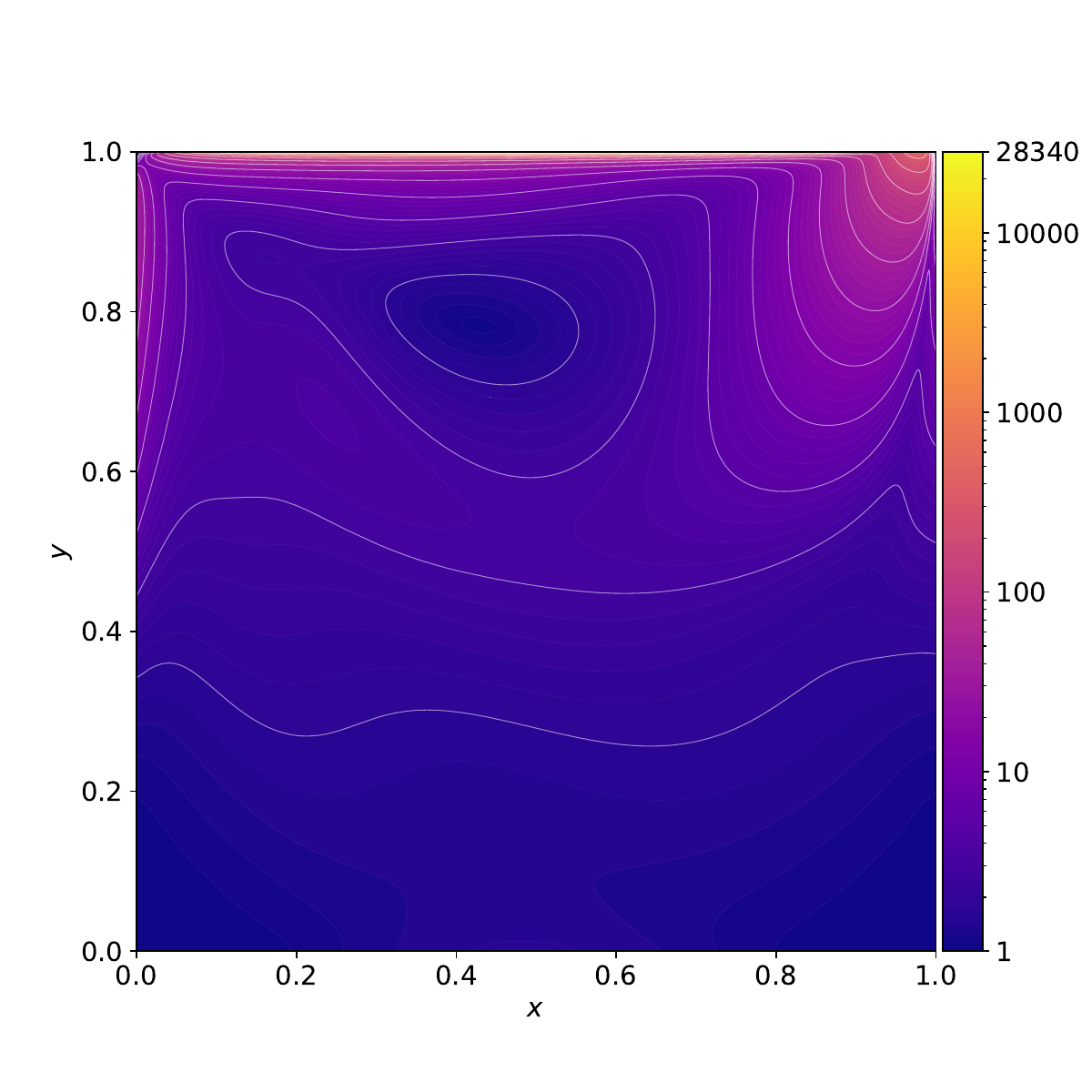} 
  \caption{\label{fig:eigenvalues} Left: Smallest eigenvalue of the discrete
  conformation tensor. Right: largest eigenvalue. Both plots are logarithmically
  coloured and contoured. Minimum and maximum values are shown on the colour
  bar.}
  \end{figure}

\section{Conclusions \& Discussion}\label{sec:discuss}

Results from simulations of an Oldroyd-B fluid in the creeping flow regime using
a discretisation of the Lie derivative \eqref{eq:Lie_derivative_def_1} were
presented. The scheme maintained positive definiteness of the discrete
conformation tensor and achieved good agreement with existing published data.
Our proposed finite difference method circumvents quadrature issues,
representing advection effectively compared to some finite element schemes,
which can under-represent advection for small time steps or coarse quadrature
formulas \cite{boyaval2009free}.

Qualitatively and quantitatively, the results are consistent with those in the
literature using the log-conformation representation
\cite{fattal2005time,pan2009simulation,sousa2016lid}. The scheme's update
process for the conformation tensor is efficient and straightforward to
implement. However, spatial resolution and mesh design significantly influence
the results, requiring fine meshes for convergence.

Further work includes exploring enhancements to preserve additional
fluid structures, particularly incompressibility constraints and
investigating differential equations on the special linear group for
a consistent discrete deformation gradient. This may involve
combining the approach with divergence-free finite element methods
for fluid equations or exploring finer meshes or different finite
difference points for the constitutive law to address challenges in
the creeping flow regime.  

\section{Acknowledgements}

This work has been partially supported by the Leverhulme Trust
Research Project Grant No. RPG-2021-238.  TP is also partially
supported by EPRSC grants
\href{https://gow.epsrc.ukri.org/NGBOViewGrant.aspx?GrantRef=EP/W026899/2}
     {EP/W026899/2},
     \href{https://gow.epsrc.ukri.org/NGBOViewGrant.aspx?GrantRef=EP/X017206/1}
          {EP/X017206/1} and
          \href{https://gow.epsrc.ukri.org/NGBOViewGrant.aspx?GrantRef=EP/X030067/1}
               {EP/X030067/1}. The authors also want to thank Gabriel
               Barrenechea and Emmanuil Geourgoulis for helpful
               discussions and suggestions.

\printbibliography
\end{document}